\documentclass{article}

\usepackage[a-2u]{pdfx}

\usepackage{lmodern}
\usepackage{iftex}
\ifpdftex
\usepackage[utf8]{inputenc}
\usepackage[T1]{fontenc}
\usepackage{textcomp}
\fi

\usepackage{amsmath}        
\usepackage{amssymb}        
\usepackage{amsfonts}       
\usepackage{amsthm}         
\usepackage{bbding}         
\usepackage{bm}             
\usepackage{graphicx}       
\usepackage{fancyvrb}       
\usepackage{indentfirst}    
\usepackage[nottoc]{tocbibind} 
\usepackage{icomma}         
\usepackage{dcolumn}        
\usepackage{booktabs}       
\usepackage{paralist}       
\usepackage{xcolor}         

\usepackage[explicit]{titlesec}

\usepackage[nopatch=item]{microtype}   

\usepackage{mathtools}   

\usepackage{pgfplots}
\pgfplotsset{compat=1.15}
\usepackage{mathrsfs}
\usetikzlibrary{arrows}
\usepackage{circuitikz}

\usepackage{float}

\usepackage{url}
\usepackage{subcaption}

\newenvironment{idea}{
\par \noindent
  \textit{Idea of proof}.
}{
\newline
}

\newcommand{\R}{\mathbb{R}}
\newcommand{\N}{\mathbb{N}}

\newcommand{\Q}{\mathbb{Q}}

\DeclareMathAlphabet{\pazocal}{OMS}{zplm}{m}{n}
\newcommand{\aex}[1]{\mathcal{L}^n\text{-a.e. } x_0 \in #1}
\newcommand{\Wkp}[3]{W^{#1,#2}(#3)}
\newcommand{\Lp}[2]{L^{#1}({#2})}
\newcommand{\Ll}{\pazocal{L}}

\newcommand{\h}[1]{\widehat{#1}\phantom{}}

\newcommand{\Cc}[1]{\pazocal{C}_c\left(#1\right)}

\newcommand{\degree}[3]{\operatorname{deg}\left(#1, #2, #3\right)}
\newcommand{\fOm}{f: \Omega \to \R^n}
\newcommand{\B}[1]{\mathbb{B}_{#1}}
\newcommand{\Ss}[1]{\mathbb{S}_{#1}}
\newcommand{\Bbar}[1]{\overline{\mathbb{B}}_{#1}}
\newcommand{\A}{\mathbb{A}}
\newcommand{\e}{\mathbf{e}}
\newcommand{\Taex}[1]{\mathcal{L}^{3}\text{-a.e. } x_0 \in #1}
\newcommand{\aey}[1]{\mathcal{L}^{#1}\text{-a.e. }}
\renewcommand{\H}[1]{\mathcal{H}^{#1}}
\newcommand{\testfunctions}[2]{\phi \in \mathcal{C}_c^1(#1, \R^{#2}), ||\phi||_\infty \leq 1}
\newcommand{\parc}[2]{\frac{\partial #1}{\partial #2}}
\newcommand{\sgn}{\operatorname{sgn}}
\newcommand{\spt}{\operatorname{spt}}
\newcommand{\diam}{\operatorname{diam}}

\newcommand{\dist}{\operatorname{dist}}
\newcommand{\abs}[1]{\left\lvert{#1}\right\rvert}
\newcommand{\DfX}[1]{\langle Df, #1 \rangle}

\def\Xint#1{\mathchoice
   {\XXint\displaystyle\textstyle{#1}}%
   {\XXint\textstyle\scriptstyle{#1}}%
   {\XXint\scriptstyle\scriptscriptstyle{#1}}%
   {\XXint\scriptscriptstyle\scriptscriptstyle{#1}}%
   \!\int}
\def\XXint#1#2#3{{\setbox0=\hbox{$#1{#2#3}{\int}$}
     \vcenter{\hbox{$#2#3$}}\kern-.5\wd0}}

\def\dashint{\Xint-}

\theoremstyle{plain}
\newtheorem{theorem}{Theorem}[section]
\newtheorem{lemma}[theorem]{Lemma}
\newtheorem{proposition}[theorem]{Proposition}
\newtheorem{cor}[theorem]{Corollary}
\newtheorem{remark}[theorem]{Remark}

\newtheorem{definition}[theorem]{Definition}

\theoremstyle{remark}

\usepackage{chngcntr}
\counterwithin{equation}{section}

\title{Jacobians of BV homeomorphisms and their weak-* limits}
\author{Jakub Petr\thanks{The author was supported by the grant GA\v{C}R P201/24-10505S.}}
\date{}

\begin{document}
\maketitle
\noindent \textbf{Abstract:} We prove that in $\R^3$ any $BV$ homeomorphism as well as the weak-* limit of $BV$ homeomorphisms cannot change the sign of the Jacobian, i.e., we have $J_f \geq 0$ almost everywhere or $J_f \leq 0$ almost everywhere.

\section{Introduction}\label{sec1}

Our main goal is to prove that the sign of the Jacobian is preserved under the weak-* convergence in the class of homeomorphisms in $BV(\Omega, \R^3)$. This question naturally arises in the variational approach to Geometric Function Theory (\cite{AstalaIwaniecMartin2009}, \cite{IwaGav2001}, \cite{Reshetnyak89}) and Nonlinear Elasticity (\cite{Antman95}, \cite{Ball76}, \cite{Ciarlet1994}, \cite{Marsden1994}, \cite{Truesdell2004}, \cite{Silhavy1997}). Both theories often deal with minimizing sequences of Sobolev homeomorphisms. In the context of NE, one typically deals with two- or three-dimensional models and requires that the mapping does not change the orientation, i.e., the deformation gradients belong to $$M^{n \times n}_+ = \{A: \text{real }n \times n \text{ matrices with } \det A > 0\}.$$ In general, the infimum of the energy is not attained by a homeomorphism, as compression of the material can occur. This motivates our interest in the weak-* limits of homeomorphisms.

For a smooth diffeomorphism $f: \Omega \to \R^n$, we know that $J_f > 0$ or $J_f < 0$ on the whole of a domain $\Omega \subset \R^n$. The authors of \cite{HenMal2010} showed that when $1 \leq p < \infty$ for $n = 2,3$ or $p > \lfloor\frac{n}{2}\rfloor$ for $n \geq 4$, the sign of the Jacobian of homeomorphisms belonging to the Sobolev class $W^{1,p}$ cannot change. Later, under the same constraints for $p$, the authors of \cite{HenclOnninen2018} showed that the weak limit $f$ of a sequence of homeomorphisms $f_k \in \Wkp{1}{p}{\Omega, \R^n}$ with $J_{f_k} > 0$ on a set of positive measure satisfies $J_f \geq 0$ almost everywhere. It was also shown that the Jacobian of a Sobolev homeomorphism can change its sign in dimension $n \geq 4$ for $1 \leq p < \lfloor\frac{n}{2}\rfloor$ (\cite{CamHenTen18}, \cite{HenVej16}). The limiting case $p = \lfloor\frac{n}{2}\rfloor$ is yet to be solved in full generality. For partial results, see \cite{GoldHaj19}.

However, for $p=1$ the Sobolev space $W^{1,1}$ is not reflexive. Thus, it is natural to substitute the role of $W^{1,1}$ with $BV$, where one has weak-* compactness of the unit ball, for applications in Calculus of Variations.

The sign of the Jacobian is an analytical indicator of the orientation of a given mapping. From another point of view, every homeomorphism is sense-preserving or sense-reversing. Our first result establishes the correspondence between the topological and the analytical notion of orientation for $BV$ homeomorphisms. This result implies that any $BV$ homeomorphism in $\R^3$ cannot change the sign of the Jacobian.

\begin{theorem}
\label{thm: BV jacobian}
Let $n \in \{2,3\}$ and $\Omega \subset \R^n$ be a domain. Let $f \in BV(\Omega, \R^n)$ be a sense-preserving homeomorphism. Then for $\aex{\Omega}$ we have $J_f(x_0) \geq 0$.
\end{theorem}

Our main result states that the sign of the Jacobian cannot change even in the case of weak-* limits of $BV$ homeomorphisms.

\begin{theorem}
\label{thm: BV limit jacobian}
Let $n \in \{2, 3\}$ and $\Omega \subset \R^n$ be a domain. Let $f_k \in BV(\Omega, \R^n)$ be homeomorphisms satisfying $J_{f_k} > 0$ on a set with positive measure and let $f \in BV(\Omega, \R^n)$ be the weak-* limit of $\{f_k\}$ in $BV$. Then for $\aex{\Omega}$ we have $J_f(x_0) \geq 0$.

Moreover, in the case $n=2$, $f$ is differentiable almost everywhere in $\Omega$.
\end{theorem}

Alternatively, it would be enough in Theorem \ref{thm: BV limit jacobian} to assume that all $f_k$'s are sense-preserving. It does not suffice to assume $J_{f_k} \geq 0$ almost everywhere in $\Omega$. In fact, the authors of \cite{HenclOnninen2018} mention that even in the Sobolev setting $\Wkp{1}{1}{\Omega, \R^3}$, using the construction in \cite{Hencl2011}, there exists a sequence of (sense-reversing) homeomorphisms $f_k$ satisfying $J_{f_k} = 0$ almost everywhere in $\Omega$ and converging weakly in $\Wkp{1}{1}{\Omega, \R^3}$ to $f(x) = (-x_1, x_2, x_3)$.

In Section \ref{chapter: intuition}, we start by giving the geometric intuition of our proofs of Theorems \ref{thm: BV jacobian} and \ref{thm: BV limit jacobian}. In Chapter \ref{chapter: prelim} we continue by recalling the basic tools we use. Next, we prove Theorem \ref{thm: BV jacobian} in Chapter \ref{chapter: BV}. The core of this paper is the first section of Chapter \ref{chapter: BV*}, where we provide new tools that we use in the second section to prove Theorem \ref{thm: BV limit jacobian}. The last section is a quick proof of the Moreover part of \ref{thm: BV limit jacobian}.

\subsection{Geometric intuition}
\label{chapter: intuition}
We present the intuition behind the non-negativity of the Jacobian in the case $n=3$. 

\subsubsection{Intuition behind the proof of Theorem \ref{thm: BV jacobian}}

Assume, for contradiction, that the conclusion is not valid. This means that we can find a point $x_0 \in \Omega$ where the Jacobian of $f$ is negative. After a linear change of variables, we can assume, without loss of generality, that the absolutely continuous part of the weak derivative $Df(x_0)$ is as in \eqref{eq: BV Df(x0)}. Moreover, by translating to the origin, we can take $x_0 = 0$. Then we consider a link consisting of a pair of small one-dimensional circles near the point $0$. We choose the orientation of the link so that the \emph{linking number} is equal to $1$. Then we apply the mapping $f$. As the approximate derivative $\nabla f(0)$ is equal to the absolutely continuous part of the weak derivative $Df(0)$, we know that on a large portion of the link the mapping $f$ behaves almost like a reflection with respect to the plane perpendicular to the $z$ axis.
\begin{center}
\includegraphics[scale=0.6]{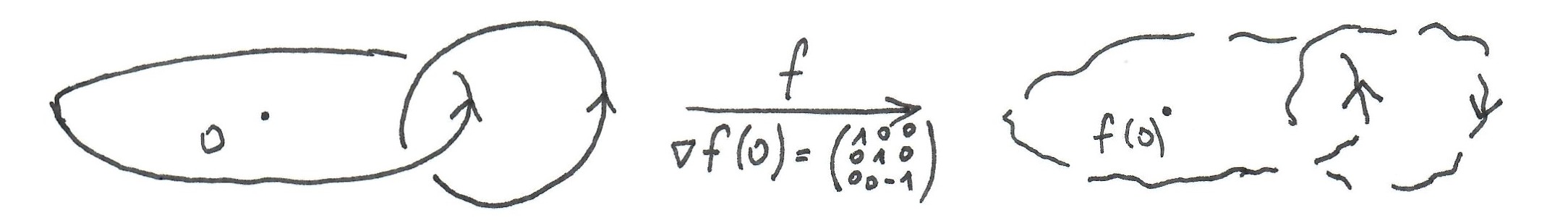}    
\end{center}
Hence, it seems that the linking number of the image of the link should be equal to $-1$. However, we know that $f$ is sense-preserving. This means that the linking number of the image has to be equal to $+1$! In order for this to occur, there must be a set of small $\H{1}$ measure in which the mapping $f$ oscillates significantly.
\begin{center}
\includegraphics[scale=0.6]{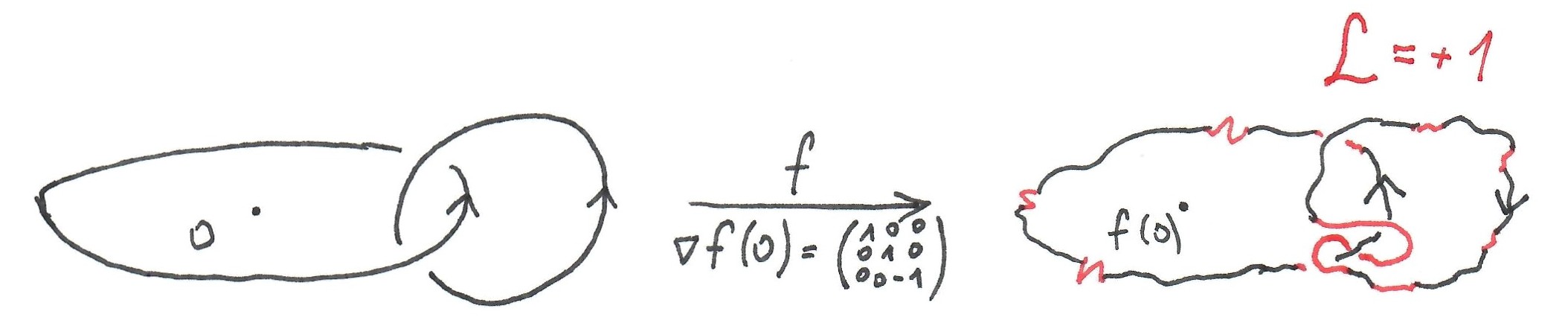}
\end{center}
Therefore, if we consider many similar links around $0$ and integrate the oscillation over such links, we obtain a set of small $\Ll^3$ measure where the mapping $f$ oscillates highly. Once we estimate this oscillation in Chapter \ref{chapter: BV}, we reach a contradiction with the fact that $\nabla f$ is in $L^1$.

\subsubsection{Intuition behind the proof of Theorem \ref{thm: BV limit jacobian}}

Once again, for contradiction, we assume that the conclusion is not valid. Consider the matrix $R$ as in \eqref{eq: BV* matrix I}. After a linear change of variables, we assume that there is a set of positive measure $E \subset \Omega$ where for every $x \in E$ we have that the absolutely continuous part of $Df(x)$ is, without loss of generality, almost $R$. In \textbf{Step 1} we find a good density point $x_0 \in E$ and look at a small cube around $x_0$. We assume $x_0 = 0$ and imagine that this cube is filled with links, each consisting of a pair of squares. As the approximate derivative $\nabla f(0)$ is almost $R$, we get a similar picture as above -- on the vast majority of the links on a large portion of the squares, the mapping $f$ behaves almost like a linear mapping $R$. Then we consider the homeomorphisms $f_k$. As $J_{f_k} > 0$ on a set of positive measure, we conclude, due to Theorem \ref{thm: BV jacobian}, that every $f_k$ must be sense-preserving. Moreover, as $f_k \xrightharpoonup[]{*} f$ in $BV$, in particular $f_k \to f$ in $L^1$, it follows that $f_k \to f$ in measure. Hence, for $k \in \N$ large enough, we find that on the vast majority of the links and on a large portion of the squares, the mappings $f_k$ behave almost like the mapping $f$. Now, $f$ is close to the mapping $R$, the mappings $f_k$ are close to $f$, which means that they seem to reverse the orientation. However, $f_k$'s are sense-preserving. This means that for every $k \in \N$ large enough on the vast majority of the links, there is a set of small $\H{1}$ measure where the mapping $f_k$ oscillates significantly. Without loss of generality, we pick one of the $8$ sides of the $2$ squares where the mappings $f_k$ oscillate highly on a large portion of the links. This idea is formally presented in \textbf{Step 2} and proven in \textbf{Step 3}.
\begin{center}
\includegraphics[scale=0.45]{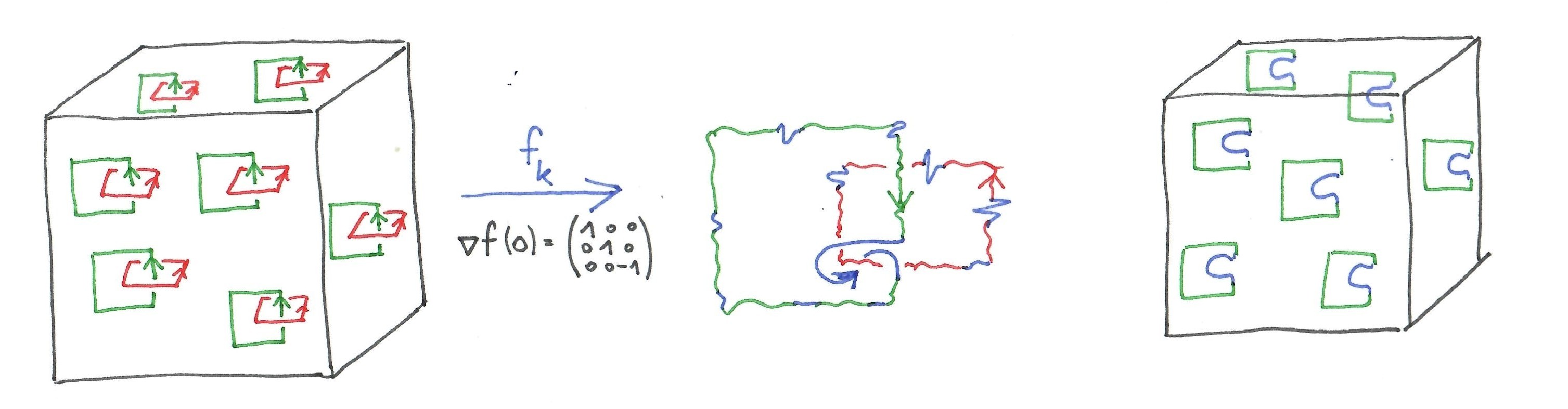}
\end{center}
In \textbf{Step 1}, we pick our center point of the cube $x_0 = 0$ so that we can consider progressively smaller side lengths of the squares and obtain the same picture. The reason we do that is because we want to use Lemma \ref{lem: Mountain Range}, which roughly states the following. Suppose that we have a one-dimensional real function $g$ defined on an interval such that it oscillates wildly on many dyadic scales $2^{-k}$. Then the total variation of $g$ must be $\infty$. Next comes Lemma \ref{lem: BV limit Dgk unbounded} which serves as a bridge between the one-dimensional Lemma \ref{lem: Mountain Range} and the three-dimensional setting of Theorem \ref{thm: BV limit jacobian}. Finally in \textbf{Step 4}, we return to our cube filled with square links on many dyadic scales and find a set $\h{F}_Q$ of positive $\H{2}$ measure consisting of vertical lines inside the cube, where a lot of oscillation happens and where we can apply Lemma \ref{lem: BV limit Dgk unbounded}.
\begin{center}
\includegraphics[scale=0.45]{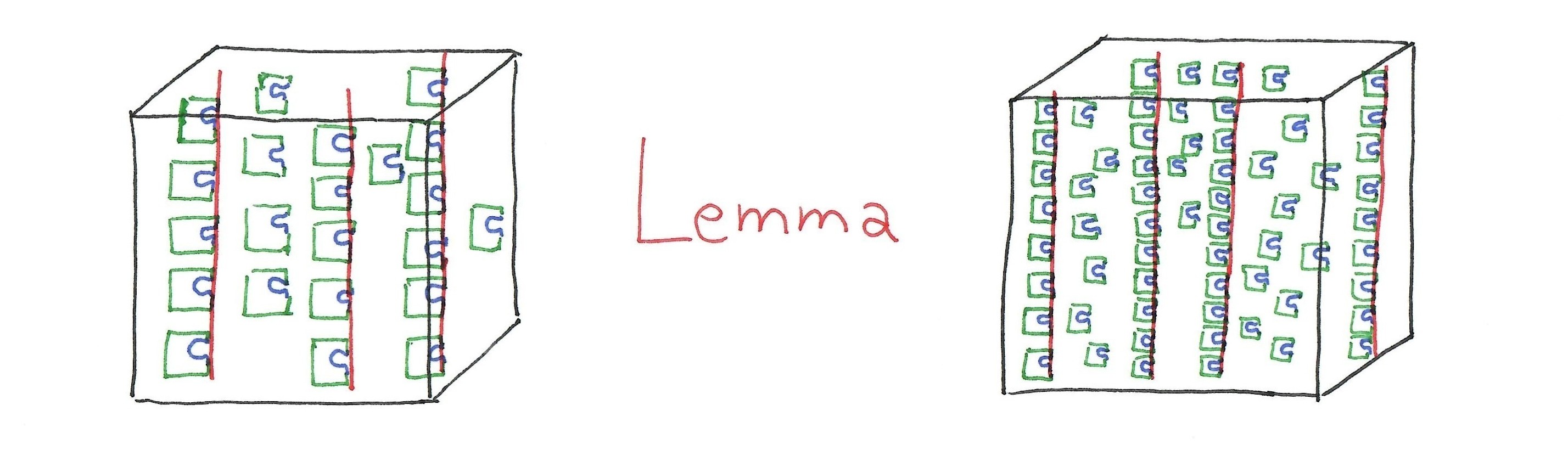}
\end{center}
We infer that the total variations of $Df_k$ satisfy $\sup_{k \in \N} |Df_k|(\Omega) = \infty$. However, this contradicts the fact that $f_k \xrightharpoonup[]{*} f$ in $BV$ and, by the Uniform Boundedness Principle, we have $\sup_{k \in \N} |Df_k|(\Omega) < \infty$. This concludes \textbf{Step 5} as well as the entire proof.

\section{Preliminaries}
\label{chapter: prelim}
Throughout the text, we assume that $\Omega \subset \R^n$ is a domain.

We denote by $\Cc{\Omega, \R^m}$ the set of all mappings $f: \Omega \to \R^m$ with a compact support in $\Omega$. Furthermore, by $\Lp{p}{\Omega, \R^n}$ we mean the space of functions $\fOm$ such that for every $\iota \in \{1, \ldots, n\}$ the coordinate functions satisfy $f^\iota \in \Lp{p}{\Omega}$. Here, $\Lp{p}{\Omega}$ denotes the standard Lebesgue space with respect to the Lebesgue measure $\mathcal{L}^n$ restricted to $\Omega$.

We denote the unit ball in $\R^d$ by $\B{d}$ and the unit sphere by $\Ss{d-1}$. In particular, $\Ss{1}$ denotes the unit circle in $\R^2$. Moreover, by $\Bbar{d}(c,r)$ we denote the closed ball in $\R^d$ with center $c \in \R^d$ and radius $r>0$.
For $c \in \R^3$ and $r > 0$ we denote $Q(c, r) \subset \R^3$ the cube centered at $c$ with side length $r$. Furthermore, $s(Q)$ denotes the center of the cube $Q \subset \R^3$ and we denote its upper face by $h(Q) \subset \R^2$. 
Let $(x^*, y^*) \in \R^2$. We denote by $P_{x^*, y^*} \subset \R^3$ the line such that every $(x, y, z) \in P_{x^*, y^*}$ satisfies both $x = x^*$ and $y = y^*$.

We use the symbol $C$ for a constant that may change its value from line to line (if not specified).

The symbol $\lesssim$ denotes an inequality up to a multiple. If $x \lesssim y$ and $y \lesssim x$, we write $x \simeq y$.

\subsection{Degree}

Let $\Omega \subset \R^n$ be open. For a continuous mapping $f: \Omega \to \R^n$ and $y_0 \in \R^n \setminus f(\partial\Omega)$, the \emph{degree of $f$ at $y_0$ with respect to $\Omega$} is denoted by $\degree{f}{y_0}{\Omega}$. We can intuitively think of degree as a function that counts the number of preimages of $y_0$ in $\Omega$ ``with orientation'' under the mapping $f$. We find that if $\Omega \subset \R^n$ is a domain and $\fOm$ is a homeomorphism, then $\degree{f}{y_0}{\Omega}$ is constantly 1 or -1 in $f(\Omega)$. For further reference, see \cite{FonsecaGangbo95}.

\begin{definition}
Let $n \in \N$ and $\Omega \subset \R^n$ be a domain. Let $f: \Omega \to \R^n$ be a homeomorphism. If for every $x \in \Omega$ we have $\deg(f, f(x), \Omega) = 1$, we call such a homeomorphism \emph{sense-preserving}. If for every $x \in \Omega$ we have $\deg(f, f(x), \Omega) = -1$, we call it \emph{sense-reversing}.
\end{definition}

Further, we need the following properties of the degree.

\begin{proposition}[Degree for linear mappings, \cite{FonsecaGangbo95}]
\label{prop: Degree for linear mappings}
Let $A: \R^n \to \R^n$ be a linear mapping with $\det A \neq 0$. Then $$\degree{A}{y_0}{\Omega} = \sgn \det A.$$
\end{proposition}

\begin{proposition}[Degree for smooth homeomorphisms, \cite{FonsecaGangbo95}]
\label{prop: Degree for smooth homeomorphisms}
Let $\fOm$ be a homeomorphism and suppose that $f$ is differentiable at $x_0 \in \Omega$, $J_f(x_0) \neq 0$. Then $$\degree{f}{f(x_0)}{\Omega} = \sgn J_f(x_0).$$
\end{proposition}

\begin{proposition}[Stability under homotopy, \cite{FonsecaGangbo95}]
\label{prop: Stability under homotopy}
Let $H: \Omega \times [0,1] \to \R^n$ be a continuous mapping and $y_0 \in \R^n \setminus H(\partial\Omega,t)$ for every $t \in [0,1]$. Then $$\degree{H(\cdot, 0)}{y_0}{\Omega} = \degree{H(\cdot, 1)}{y_0}{\Omega}.$$
\end{proposition}

\subsection{Linking number}

We use the definition of the linking number from \cite{HenMal2010}.

Consider the mapping $\Phi(\xi, \eta): \Bbar{2}\times\Bbar{2} \to \R^3$ defined coordinate-wise as $\Phi(\xi, \eta)=x$ where
\begin{align*}
x_1&=(2+\eta_1)\xi_1,\\
x_{2}&=(2+\eta_1)\xi_{2},\\
x_{3}&=\eta_{2}.
\end{align*}

We denote by $\A$ the interior space enclosed by a torus, i.e., $$\A = \Phi(\Ss{1}\times\B{2}) = \left\{x \in \R^3: \left(\sqrt{x_1^2+x_{2}^2}-2\right)^2+x_{3}^2<1 \right\}.$$ Given $x \in \overline{\A}$, we can find 
\begin{equation}
\label{eq: link definition}
\text{unique }\xi \in \Ss{1} \text{ and } \eta \in \Bbar{2} \text{ such that } \Phi(\xi, \eta) = x.
\end{equation} We denote these by $\xi(x)$ and $\eta(x)$.

A \emph{link} is a pair $(\varphi, \psi)$ of closed curves $\varphi, \psi: \Ss{1} \to \R^3$. The \emph{linking number} of the link $(\varphi, \psi)$ is defined as the degree $$\mathcal{L}(\varphi, \psi) = \degree{L}{0}{\A},$$ where the mapping $L: \overline{\A} \to \R^3$ is defined as $$L(x) = \varphi(\xi(x))-\Bar{\psi}(-\eta(x)),$$ or equivalently
\begin{equation}
\label{def: linking number}
    L(\Phi(\xi, \eta))=\varphi(\xi)-\Bar{\psi}(-\eta),
\end{equation}
where $\Bar{\psi}$ is an arbitrary continuous extension of $\psi$ to $\Bbar{2}$ (it is well-known that the degree depends only on the values on the boundary $\partial\A=\Phi(\Ss{1}\times\Ss{1})$). Intuitively, we can think of the linking number as the number of loops of a curve $\varphi$ around a curve $\psi$ counting the orientation as $+1$ or $-1$.

The \emph{canonical link} is the pair $(\mu, \nu)$ where
\begin{align*}
    \mu(\xi)&=\Phi(\xi, 0),\\
    \nu(\eta)&=\Phi(\e_1, \eta),
\end{align*}
for $\xi, \eta \in \Ss{1}$. In fact, we get
\begin{align*}
    \mu(\Ss{1})&=\left\{x \in \R^3: x_3=0, x_1^2+x_2^2 = 4\right\},\\
    \nu({\Ss{1}})&=\left\{x \in \R^3: x_2=0, (x_1-2)^2+x_3^2=1\right\}.
\end{align*}
\begin{figure}[H]
    \centering
    \includegraphics[width = 0.5\textwidth]{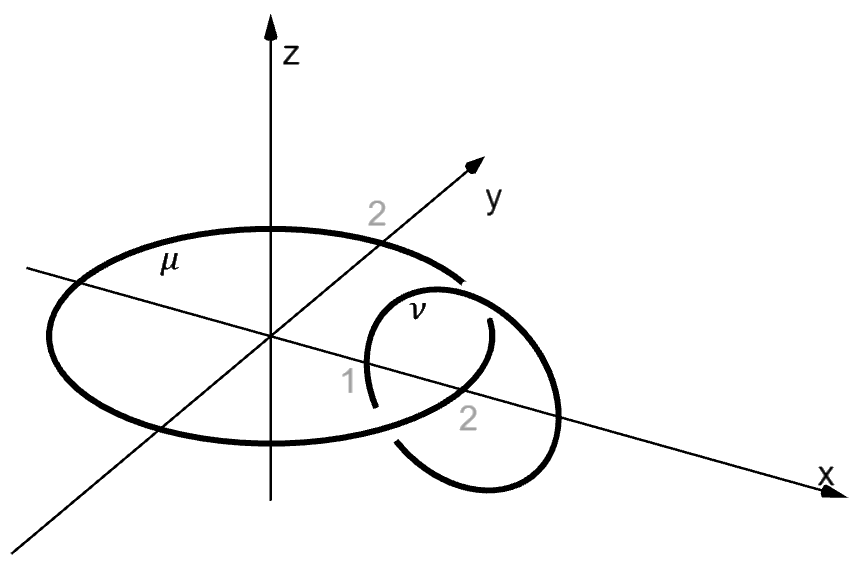}
    \caption{The canonical link in $\R^3$ (made with \href{https://www.geogebra.org/3d}{Geogebra 3D})}
\end{figure}
Observe that the canonical link consists of two circles at a unit distance.

\begin{remark}
\label{rem: canonical link distance}
For every $\xi, \eta \in \Ss{1}$, we have $$\abs{\mu(\xi) - \nu(\eta)} \geq 1.$$
\end{remark}

It is well known that the linking number is a topological invariant. To be more precise, the following proposition holds. The proof can be found in \cite{HenMal2010}.

\begin{proposition}
\label{prop: linking number canonical}
Let $f: \B{3}(0, 4) \to \R^3$ be a homeomorphism. Then the linking number $\mathcal{L}(f \circ \mu, f \circ \nu)$ is $1$ if $f$ is sense-preserving and $-1$ if $f$ is sense-reversing.
\end{proposition}

Analogously, for $a \in \Bbar{2}(0, \frac{1}{10})$ and $b \in \Bbar{2}(\e_1, \frac{1}{10}) \cap \Bbar{2}$ we can consider a \emph{perturbed canonical link} $(\mu_a, \nu_b)$ defined as
\begin{equation}
\label{eq: perturbed canonical links}
\begin{split}
    \mu_a(\xi) &= \Phi(\xi, a),\\
    \nu_b(\eta) &= \Phi(b, \eta),
\end{split}
\end{equation}
where $\xi, \eta \in \Ss{1}$. To justify the adjective ``perturbed'', let us observe that the distance of $\mu$ from $\mu_a$ as well as of $\nu$ from $\nu_b$ is relatively small.
\begin{remark}
\label{rem: distance of mu_a and mu}
Let $a \in \Bbar{2}(0, \frac{1}{10})$ and $\xi \in \Ss{1}$. Then $$\abs{\mu(\xi)-\mu_a(\xi)} \leq \frac{1}{10}.$$
\end{remark}

\begin{remark}
\label{rem: distance of nu_b and nu}
Let $b \in \Bbar{2}(\e_1, \frac{1}{10}) \cap \Bbar{2}$ and $\eta \in \Ss{1}$. Then $$\abs{\nu(\eta)-\nu_b(\eta)} \leq \frac{3}{10}.$$
\end{remark}

Similarly to Proposition \ref{prop: linking number canonical}, we have the following.

\begin{proposition}
\label{prop: linking number perturbed}
Let $a \in \Bbar{2}(0, \frac{1}{10})$ and $b \in \Bbar{2}(\e_1, \frac{1}{10}) \cap \Bbar{2}$. Let $f: \B{3}(0, 4) \to \R^3$ be a homeomorphism. Then the linking number $\mathcal{L}(f \circ \mu_a, f \circ \nu_b)$ is $1$ if $f$ is sense-preserving and $-1$ if $f$ is sense-reversing.
\end{proposition}

Of course, the use of the linking number is not limited to smooth curves. In Chapter \ref{chapter: BV*}, we work with a pair of linked squares. More precisely, for a point $p = (p_x, p_y, p_z) \in \R^3$ and a positive number $r > 0$ we define a link
$$L^r_p = S^r_p \cup T^r_p$$
consisting of a pair of linked squares
\begin{align*}
S^r_p &= a_p^r \cup b^r_p \cup c^r_p \cup d_p^r,\\
T^r_p &= e_p^r \cup f_p^r \cup g_p^r \cup h_p^r
\end{align*} with a positive orientation where
\begin{equation}
\label{eq: Square link sides}
\begin{aligned}
a_p^r &= \{(p_x, p_y, t): t \in [p_z - \frac{r}{2}, p_z + \frac{r}{2}]\},\\
b_p^r &= \{(t, p_y, p_z + \frac{r}{2}): t \in [p_x - r, p_x]\},\\
c_p^r &= \{(p_x - r, p_y, t): t \in [p_z - \frac{r}{2}, p_z + \frac{r}{2}]\},\\
d_p^r &= \{(t, p_y, p_z - \frac{r}{2}): t \in [p_x - r, p_x]\},\\
e_p^r &= \{(t, p_y - \frac{r}{2}, p_z): t \in [p_x - \frac{r}{2}, p_x + \frac{r}{2}]\},\\
f_p^r &= \{(p_x + \frac{r}{2}, t, p_z): t \in [p_y - \frac{r}{2}, p_y + \frac{r}{2}]\},\\
g_p^r &= \{(t, p_y + \frac{r}{2}, p_z): t \in [p_x - \frac{r}{2}, p_x + \frac{r}{2}]\},\\
h_p^r &= \{(p_x - \frac{r}{2}, t, p_z): t \in [p_y - \frac{r}{2}, p_y + \frac{r}{2}]\}.
\end{aligned}
\end{equation}
If the side length $r$ is clear from the context, we omit the upper index and simply write $a_p = a_p^r$ and so on. The square $S^r_p$ lies in the plane perpendicular to the axis $y$ and the square $T^r_p$ belongs to the plane perpendicular to the axis $z$. Moreover, we constructed the link $L^r_p$ in a way that
\begin{equation}
\label{eq: Distance Linked Squares}
    \dist (S^r_p, T^r_p) = \frac{r}{2}.
\end{equation}

\begin{figure}[H]
\centering
\includegraphics[width = 0.4\textwidth]{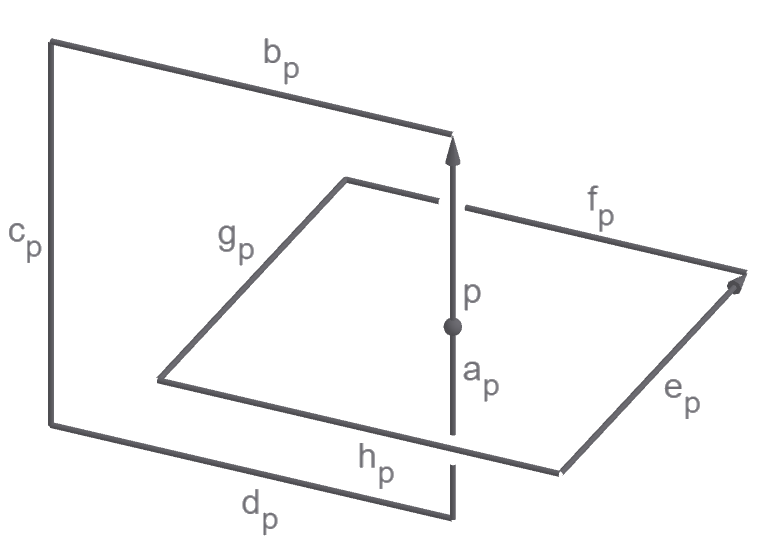}
\caption{The link $L_p^r$ with a chosen orientation (made with \href{https://www.geogebra.org/3d}{Geogebra 3D})}
\end{figure}

We now give an analogy of Proposition \ref{prop: linking number canonical} in the case of linked squares in $\R^3$. We did not rigorously define the linking number for a pair of squares, but it is easy to see that, because of a homotopy argument, the properties are exactly the same as for a pair of circles.

\begin{proposition}
\label{prop: Linking number squares}
    Let $\Omega \subset \R^3$ be an open set and suppose that $f: \Omega \to \R^3$ is a homeomorphism. In addition, let $p \in \Omega$ and $r > 0$ satisfy $L^r_p \subset \Omega$. Then the linking number $\mathcal{L}(f \circ S^r_p, f \circ T^r_p)$ is equal to $1$ if $f$ is sense-preserving and it is equal to $-1$ if $f$ is sense-reversing.
\end{proposition}

\subsection{BV}
Let $I \subset \R$ be a closed interval and $g: I \to \R$. Let $D = \{t_r\}_{r = 1}^{N_D}$ be a \emph{partition of the interval $I$}, i.e., a finite sequence of points in the interval $I$ satisfying $t_{r-1} \leq t_r$ for every $r \in \{2, \ldots, N_D\}$. We denote \emph{the variation of $g$ over $D$} as $$V(g, D) := \sum_{r=2}^{N_D} |g(t_r) - g(t_{r-1})|.$$ If $g \in BV(I)$, we know that $$|Dg|(I) = \sup\{V(g,D): D \text{ is a partition of } I\}.$$

In the following, we consider $f \in BV(\Omega, \R^n)$ for $\Omega \subset \R^n$ open. We denote the weak derivative of $f$ by $Df$, which is a $\R^{n \times n}$-valued Radon measure defined on $\Omega$. The total variation of $Df$ in $\Omega$ is denoted by $|Df|(\Omega)$. From \cite[Proposition 3.6.]{BVbible} we get that
\begin{equation}
\label{eq: variation of Du}
\abs{Df}(\Omega) = \sup \left\{\sum_{\alpha = 1}^n \int_\Omega f^{\alpha} \text{ div}\phi^\alpha \, d\Ll^n; \testfunctions{\Omega}{n \times n} \right\}.
\end{equation}

We say that a sequence of Radon $\R^m$-valued measures $\{\mu_j\}$ on $\Omega$ converges \emph{weakly-*} to a $\R^m$-valued Radon measure $\mu$ if for every $\phi \in \mathcal{C}_c(\Omega, \R^m)$ we have $$\int_\Omega \phi \, d\mu_j \to \int_\Omega \phi \, d\mu.$$ We write $\mu_j \overset{\ast}{\rightharpoonup} \mu$ in measures. Further, a sequence $\{f_j\} \subset BV(\Omega, \R^n)$ converges \emph{weakly-*} to a mapping $f \in BV(\Omega, \R^n)$ if $f_j \to f$ in $\Lp{1}{\Omega, \R^n}$ and $Df_j \overset{\ast}{\rightharpoonup} Df$ in measures. We write $f_j \overset{\ast}{\rightharpoonup} f$ in $BV$.

We can decompose the weak derivative of $f$ into the absolutely continuous and the singular part with respect to $\Ll^n$, that is, $$Df=D^af+D^sf.$$ We write $D^af = g \mathcal{L}^n$, where $g \in \Lp{1}{\Omega, \R^{n \times n}}$ is the Radon-Nikodým derivative. Then for $\aey{n} x \in \Omega$ we denote $J_f(x) = \det g(x)$ and $\nabla f(x) = g(x)$.

The following theorem is a corollary of \cite[Theorem 6.1.]{EvansGariepy}.

\begin{theorem}[$L^1$-differentiability]
\label{thm: L1 diff}
Let $f \in BV(\Omega, \R^n)$. Then $\aex{\Omega}$ satisfies $$\lim_{r \to 0_+}\dashint_{\B{n}(x_0, r)} \frac{\abs{f(x)-f(x_0)-\nabla f(x_0)(x-x_0)}}{r}\,dx = 0.$$
\end{theorem}

Regarding the singular part $D^sf$, as a consequence of the Lebesgue \linebreak decomposition theorem (see \cite[Theorem 1.31.]{EvansGariepy}), we obtain the following useful observation.

\begin{theorem}
\label{thm: singular measures}
Let $\mu$ be a Radon measure on $\R^n$ which is singular with respect to $\mathcal{L}^n$. Then for $\aex{\R^n}$ we have
$$\lim_{r \to 0_+} \frac{\abs{\mu}(B(x_0, r))}{\abs{B(x_0, r)}} = 0.$$
\end{theorem}

We are interested in the behavior of BV mappings along lines. For this reason, let $\sigma \in \Ss{n-1}$ represent a direction in $\R^n$. We denote the \emph{distributional derivative of $f$ along $\sigma$} by $\DfX{\sigma}$ and it is a $\R^n$-valued Radon measure $\mu$ satisfying for every $\phi \in \mathcal{C}^\infty_c(\Omega)$
$$\int_\Omega f \parc{\phi}{\sigma}\, d\Ll^n = - \int_\Omega \phi\, d\mu.$$
Similarly to \eqref{eq: variation of Du}, we can write for $\sigma \in \Ss{n-1}$
\begin{equation}
\label{eq: variation of Dunu}
\abs{\DfX{\sigma}}(\Omega) = \sup \left\{ \int_\Omega f \cdot \parc{\phi}{\sigma}; \testfunctions{\Omega}{n} \right\}.
\end{equation}
It is easy to see from \eqref{eq: variation of Du} and \eqref{eq: variation of Dunu} that
\begin{equation}
\label{eq: BV directional leq total}
    \abs{\DfX{\sigma}}(\Omega) \leq \abs{Df}(\Omega).
\end{equation}
For $\sigma \in \Ss{n-1}$, we denote by $\Omega_\sigma$ the orthogonal projection of $\Omega$ to the hyperplane perpendicular to $\sigma$. For any $y \in \Omega_\sigma$, let $\Omega_y^\sigma$ denote the section $\{t \in \R: y + t\sigma \in \Omega\}$ of $\Omega$ corresponding to $y$. Let $f_y^\sigma: \Omega_y^\sigma \to \R^n$ be the section of the function $u$ along $\Omega_y^\sigma$, that is
$$f_y^\sigma(t) = f(y + t\sigma).$$
There is no general Fubini-type theorem for Radon measures. The advantage of considering the sections is that we have the following Fubini-type result for $BV$ mappings (\cite[Theorem 3.103.]{BVbible}).

\begin{theorem}
\label{thm: BV total variation decomposition}
    Let $f \in BV(\Omega, \R^n)$ and $\sigma \in \Ss{n-1}$. Then
$$|\DfX{\sigma}|(\Omega) = \int_{\Omega_\sigma} |Df_y^\sigma|(\Omega_y^\sigma)\, dy.$$
\end{theorem}

This theorem yields an immediate corollary (\cite[Remark 3.104]{BVbible}).

\begin{theorem}[BVL characterization]
\label{thm: BVL}
Let $f \in \Lp{1}{\Omega, \R^n}$. Then the mapping $f \in BV(\Omega, \R^n)$ if and only if there exist $n$ linearly independent unit vectors $\sigma_i$ such that $f^{\sigma_i}_y \in BV(\Omega^{\sigma_i}_y, \R^n)$ for $\Ll^{n-1}$-almost every $y \in \Omega_{\sigma_i}$ and for every $i \in \{1, \ldots, n\}$ we have $$\int_{\Omega_{\sigma_i}} \abs{Df^{\sigma_i}_y}(\Omega^{\sigma_i}_y) < + \infty.$$
\end{theorem}

Oftentimes, we work with embedded manifolds, such as circles or annuli. In order that we can use Theorem \ref{thm: BV total variation decomposition}, we want to ``straighten'' the curvature. In Chapter \ref{chapter: BV}, we work with annuli $\A$ that are locally a smooth image of a cube $Q \subset \R^3$. For a bi-Lipschitz change of variables, we have the following result (see \cite[Theorem 3.16]{BVbible}).

\begin{theorem}
\label{thm: BV Lipschitz change of variables}
    Let $f \in BV(\Omega, \R^n)$, let $\Omega' \subset \R^n$ be open, let $\varphi: \Omega \xrightarrow[]{onto} \Omega'$ be bi-Lipschitz and let $L:= Lip(\varphi)$. Then $f \circ \varphi^{-1} \in BV(\Omega', \R^n)$ and $$|D(f \circ \varphi^{-1})|(B) \leq L^{n-1} |Df|(\varphi^{-1}(B))$$ for any Borel set $B \subset \Omega'$.
\end{theorem}

Recall the perturbed canonical links ($\mu_a$, $\nu_b$) from (\ref{eq: perturbed canonical links}). They are locally a smooth image of a line segment. Similarly to the distributional derivative of $f$ along $\sigma$ for $\sigma \in \Ss{n-1}$, we denote the \emph{tangential derivative of $f$ along $\mu$ (resp. $\nu$)} by $\DfX{\mu}$ (resp. $\DfX{\nu}$).
Moreover, we denote $$\A_a = \Phi\left(\Ss{1} \times \B{2}(0, \frac{1}{10})\right)$$ and $$\A_b = \Phi\left(\B{2}(\e_1, \frac{1}{10}) \cap \B{2} \times \Ss{1}\right)$$ the parts of the annulus $\A$ filled by all possible $\mu_a(\Ss{1})$ and $\nu_b(\Ss{1})$ respectively.
Analogously to (\ref{eq: BV directional leq total}) we obtain an estimate (up to the constant from Theorem \ref{thm: BV Lipschitz change of variables})
\begin{equation}
\label{eq: tangential leq total}
    \abs{\DfX{\mu}}(\A_a) \leq C\abs{Df}(\A_a),
\end{equation}
and a similar claim is true for $\abs{\DfX{\nu}}(\A_b)$ as well.
Furthermore, for $a \in \Bbar{2}(0, \frac{1}{10})$ and $\xi \in \Ss{1}$ we denote $$f^\mu_a(\xi) = f(\mu_a(\xi))$$ and analogously for $b \in \Bbar{2}(\e_1, \frac{1}{10}) \cap \Bbar{2}$ and $\eta \in \Ss{1}$ we write $$f^\nu_b(\eta) = f(\nu_b(\eta)).$$
Due to Theorem \ref{thm: BV Lipschitz change of variables}, we get a claim similar to Theorem \ref{thm: BV total variation decomposition}.
\begin{theorem}
\label{thm: BV tangential variations decomposition}
Let $f \in BV(\B{3}(0,4), \R^3)$. Then $$\abs{\DfX{\mu}(\A_a)} \simeq  \int_{\B{2}(0, \frac{1}{10})} \abs{Df^\mu_a(\Ss{1}))} d\H{2}(a)$$ and $$\abs{\DfX{\nu}(\A_b)} \simeq  \int_{\B{2}(\e_1, \frac{1}{10}) \cap \B{2}} \abs{Df^\nu_b(\Ss{1}))} d\H{2}(b).$$
\end{theorem}

Let us explicitly formulate a corollary similar to Theorem \ref{thm: BVL}.
\begin{cor}
\label{cor: BVL perturbed}
Let $f \in BV(\B{3}(0,4), \R^3)$. Then for $\H{2}$-almost every $a \in \B{2}(0,\frac{1}{10})$ and $\H{2}$-almost every $b \in \B{2}(\e_1, \frac{1}{10}) \cap \B{2}$ we have $f^\mu_a \in BV(\Ss{1}, \R^3)$ and $f^\nu_b \in BV(\Ss{1}, \R^3)$.
\end{cor}

In Chapter \ref{chapter: BV}, we use the following computation.
\begin{lemma}
\label{lem: BV final computation}
Let $x_0 \in \R^3$, $r > 0$, $f \in BV(\B{3}(x_0,4r), \R^3)$, $A$ be an open subset of $\B{3}(0,4r)$ and let us denote for $\xi \in \Ss{1}$ $$(\Hat{f_k})_a^\mu(\xi) = f_k(x_0 + r\mu_a(\xi)).$$ Further, let $\Hat{A^\mu_a} \subset \Ss{1}$ be the part of the circle $\Ss{1}$ that is mapped by $r\mu_a$ into $A$. Then there exists an absolute constant $C > 0$ such that $$\int_{\Bbar{2}\left(0, \frac{1}{10}\right)} |D(\Hat{f^\mu_a})|(\Hat{A^\mu_a}) \, d\H{2}(a) \leq C r^{-2} |Df|(x_0 + A).$$
\end{lemma}

\begin{proof}
Let us denote $\iota(x) = x_0 + rx$. First, we apply Theorem \ref{thm: BV Lipschitz change of variables} to one-dimensional functions $f^\mu_a$ with $\varphi^{-1} = \mu_a$. Furthermore, we apply Theorem \ref{thm: BV tangential variations decomposition} to get
\begin{align*}
\int_{\Bbar{2}\left(0, \frac{1}{10}\right)} |D(\Hat{f^\mu_a})|(\Hat{A^\mu_a}) \, d\H{2}(a)&\leq C \int_{\Bbar{2}\left(0, \frac{1}{10}\right)} |D(f \circ \iota)| (\mu_a(\Hat{A^\mu_a})) \, d\H{2}(a)\\
&\lesssim |\langle D (f \circ \iota), \mu\rangle|\left(\frac{A}{r} \cap \A_a\right).
\end{align*}
Next, we simply use (\ref{eq: tangential leq total}) and again Theorem \ref{thm: BV Lipschitz change of variables}, this time to three-dimensional functions $f \circ \iota$ with $\varphi = \iota^{-1}$, to obtain
\begin{align*}
\int_{\Bbar{2}\left(0, \frac{1}{10}\right)} |D(\Hat{f^\mu_a})|(\Hat{A^\mu_a}) \, d\H{2}(a)&\leq C|D(f \circ \iota)|\left(\frac{A}{r}\right)\\
&\leq C r^{-2} |Df|(x_0 + A).
\end{align*}
\end{proof}

\subsection{Differentiability in two dimensions}

The two-dimensional space has a lot of topological obstructions that manifest in various ways. One of them is the following theorem from \cite{HenOnnKos2007}.

\begin{theorem}
\label{thm: BV homeo are diff}
Let $\Omega, \Omega'$ be two domains in $\R^2$ and suppose that $f: \Omega \xrightarrow[]{onto} \Omega'$ is a homeomorphism. Then $f \in BV(\Omega, \R^2)$ if and only if $f^{-1} \in BV(\Omega', \R^2)$. Moreover, both $f$ and $f^{-1}$ are differentiable almost everywhere.
\end{theorem}

In higher dimensions, we know only that $BV$ mappings are approximately differentiable almost everywhere (see, for instance, \cite[Theorem 6.4.]{EvansGariepy}). What this means is that whenever we look at a small ball around a point, then, up to a set of an arbitrarily small measure, its image looks very much like a linearly transformed ball. Hence, Theorem \ref{thm: BV homeo are diff} says something much stronger: a mapping $f$ being differentiable at $x$ implies that a small ball around $x$ maps in its entirety almost linearly.

The above theorem uses the following lemma, also proved in \cite{HenOnnKos2007}.

\begin{lemma}
\label{lem: BV diameter of ball}
Let $\Omega \subset \R^2$ be a domain. Suppose that $f \in BV(\Omega, \R^2)$ is a homeomorphism and let $B(y, 2r) \subset \subset f(\Omega)$. Then there exists an absolute constant $C > 0$ such that
\begin{equation*}
\label{eq: BV diameter of ball}
    r \diam f^{-1}(B(y,r)) \leq C\abs{Df}(\overline{f^{-1}(B(y,2r))}).
\end{equation*}
\end{lemma}

Several years later, the authors of \cite{DOnofrioSchiattarella2013} proved that the total variation of $Df$ and $Df^{-1}$ is the same. In particular, we use the following theorem.

\begin{theorem}
\label{thm: BV total variation Df and Df^-1 are the same}
Let $\Omega, \Omega'$ be two domains in $\R^2$ and let $f: \Omega \xrightarrow[]{onto} \Omega'$ be a $BV$ homeomorphism. Then $$\abs{Df^{-1}}(\Omega') = \abs{Df}(\Omega).$$
\end{theorem}

\section{The sign of the Jacobian in $BV(\Omega, \R^n)$}
\label{chapter: BV}
Our proof of Theorem \ref{thm: BV jacobian} is inspired by the work of \cite{HenMal2010} and \cite{HenclOnninen2018}. However, since we are working in the $BV$ setting, we also need to account for the singular part of the derivative $Df$. This is realized at the beginning of the proof in \eqref{eq: BV singular part} and is applied at the end of the proof in \eqref{eq: BV final computation}.

\begin{proof}[Proof of Theorem~{\upshape\ref{thm: BV jacobian}}]
Without loss of generality, we assume $n=3$. In the case of $\Omega \subset \R^2$, $f = (f^1, f^2) \in BV(\Omega, \R^2)$, it is easy to construct a sense-preserving homeomorphism $h \in BV(\Omega \times \R, \R^3)$ by setting $$h(x_1, x_2, x_3) = (f^1(x_1,x_2), f^2(x_1, x_2), x_3)$$ whenever $(x_1, x_2) \in \Omega$ and $x_3 \in \R$. Then, clearly $J_h(x_1,x_2,x_3) \geq 0$ if and only if $J_f(x_1, x_2) \geq 0$.

Fix $\delta \in (0, \frac{1}{100})$ small enough. By Theorem \ref{thm: L1 diff} for $\Taex{\Omega}$ we can find $r_0 > 0$ such that for every $0 < r < r_0$ we have
\begin{equation}
\label{eq: BV L1 diff}
\dashint_{\B{3}(x_0, 4r)} \frac{\abs{f(x)-f(x_0)-\nabla f(x_0)(x-x_0)}}{r}\,dx < \delta^3.
\end{equation}

By de la Vallée-Poussin's theorem, we can find a suitable nondecreasing convex function $\Theta: [0, \infty) \to [0, \infty)$ such that
\begin{equation}
\label{eq: dlVP}
    \frac{\Theta(t)}{t} \text{ is increasing, } \lim_{t \to \infty} \frac{\Theta(t)}{t} = \infty \text{ and } \Theta(\abs{f}) \in \Lp{1}{\Omega}.
\end{equation}
We know that $\Taex{\Omega}$ is a Lebesgue point of $\Theta(\abs{\nabla f(x)})$, i.e.,
\begin{equation}
\label{eq: BV Theta Lebesgue point}
\lim_{r \to 0_+} \dashint_{\B{n}(x_0,4r)} \Theta(\abs{\nabla f(x)})\,dx = \Theta(\abs{\nabla f(x_0)}) < \infty.
\end{equation}

Recall that $D^sf$ denotes the singular part of the Radon measure $Df$ with respect to $\mathcal{L}^{3}$. Let $\Hat{C} > 0$ be small enough that
\begin{equation}
\label{eq: BV HatC estimate}
   C \Hat{C}\abs{\B{3}} < \frac{1}{1000}, 
\end{equation}
where $C$ is an absolute constant described in \eqref{eq: BV final computation}. Theorem \ref{thm: singular measures} gives us for $\Taex{\Omega}$ some $r_1 > 0$ such that for any $0 < r < r_1$
\begin{equation}
\label{eq: BV singular part}
\frac{\abs{D^sf}(\B{3}(x_0,4r))}{\abs{\B{3}(x_0, 4r)}} < \hat{C} \delta.
\end{equation}

Assume, for contradiction, that there exists $x_0 \in \Omega$ such that \eqref{eq: BV L1 diff}, \eqref{eq: BV Theta Lebesgue point}, and \eqref{eq: BV singular part} hold, and $J_f(x_0) < 0$. Without loss of generality, we may and do assume that $x_0=0$, $f(x_0) = f(0) = 0$ and
\begin{equation}
\label{eq: BV Df(x0)}
\nabla f(0) = \begin{pmatrix}
    1 & 0 & 0\\
    0 & 1 & 0\\
    0 & 0 & -1
\end{pmatrix}.
\end{equation}
In particular, we want the matrix in \eqref{eq: BV Df(x0)} to have a negative determinant and to represent an isometry mapping.

Our aim is to find for every small scale $r > 0$ a small set $A$ where the mapping $f$ is highly oscillating. To be more precise, we want to find $C_1 \geq 1$ such that for every $0 < r < \min\{r_0, r_1\}$ there exists $A \subset \B{3}(0, 4r)$ such that
\begin{align}
   \label{eq: BV condition 1} \abs{A} &< \delta \abs{B_3(0,4r)} \text{ and}\\
   \label{eq: BV condition 2} r^3 &\leq C_1 \int_A \abs{\nabla f}.
\end{align}
This leads to a contradiction with the fact that $\nabla f \in L^1$. Indeed, let $$C_0:=\frac{1}{C_1 4^3\abs{\B{3}}} \leq 4^3 \abs{\B{3}}.$$ 
From \eqref{eq: BV condition 1} it follows that
\begin{equation}
\label{eq: BV A}
\frac{\abs{A}}{r^3} < \delta 4^3\abs{\B{3}}.
\end{equation}
Now we are ready for a computation. We use Jensen's inequality, \eqref{eq: BV condition 2}, and in the last inequality \eqref{eq: BV A} together with the first part of \eqref{eq: dlVP}
\begin{align*}
\dashint_{\B{3}(0, 4r)} \Theta(\abs{\nabla f}) &\geq C\frac{\abs{A}}{r^3} \dashint_A \Theta(\abs{\nabla f})\\ 
&\geq C\frac{\abs{A}}{r^3} \Theta\left(\dashint_A \abs{\nabla f}\right)\\ 
&\geq C\frac{\abs{A}}{r^3} \Theta\left(\frac{r^3}{C_1\abs{A}}\right)\\
&\geq C \delta4^3 \abs{\B{3}} \Theta\left(\frac{1}{C_1\delta4^3 \abs{\B{3}}}\right)\\
&\geq C C_0\delta \Theta\left(\frac{C_0}{\delta}\right).
\end{align*}
Now passing to a limit $\delta \to 0_+$, we obtain by the second part of \eqref{eq: dlVP} $$\limsup_{r \to 0_+} \dashint_{\B{3}(0, 4r)} \Theta(\abs{\nabla f}) = \infty,$$ which contradicts \eqref{eq: BV Theta Lebesgue point}.

Thus, fix $0 < r < \min\{r_0, r_1\}$. For simplicity, we denote $$g(x):=\abs{f(x)-\nabla f(0)x}.$$ We claim that on almost every perturbed canonical link there is a point where the value of $g$ is big. In particular, we prove that for $\H{2}$-almost every $a \in \Bbar{2}(0, \frac{1}{10})$ and for $\H{2}$-almost every $b \in \Bbar{2}(\e_1, \frac{1}{10}) \cap \Bbar{2}$ we can find $\xi \in \Ss{1}$ and $\eta \in \Ss{1}$ such that
\begin{equation}
\label{eq: BV bad links}
\begin{aligned}
g(r\mu_a(\xi)) &> \frac{r}{10} \text{ or}\\
g(r\nu_b(\eta)) &> \frac{r}{10}.
\end{aligned}
\end{equation}
On the contrary, suppose that \eqref{eq: BV bad links} does not hold. We now construct a homotopy between two mappings, one of which has degree equal to $+1$ and the other to $-1$, which yields a contradiction. In order to do so, we define for $s \in [0,1]$ $$f_s(x):=(1-s)\nabla f(0)rx+sf(rx).$$ Recall \eqref{eq: link definition} and consider a homotopy $H: \overline{\A} \times [0,1] \to \R^3$ defined as $$H(\Phi(\xi, \eta),s)=(f_s\circ\mu_a)(\xi)-\overline{(f_s\circ\nu_b)}(-\eta),$$ where $\overline{(f_s\circ\nu_b)}$ denotes a continuous extension of $f_s\circ\nu_b$ from $\Ss{1}$ to $\Bbar{2}$ which also continuously depends on $s$. Such a continuous extension exists due to Tietze theorem.
Observe that $$H(\Phi(\xi, \eta), 1) = f(r\mu_a(\xi)) - f(r\nu_b(-\eta)).$$ From $f$ being sense-preserving, one can easily deduce, since $r > 0$, that $\hat{f}(x) = f(rx)$ is sense-preserving as well. By \eqref{def: linking number} and Proposition \ref{prop: linking number perturbed}, we can see that
\begin{equation}
\label{eq: BV degree s=1}
\degree{H(\cdot,1)}{0}{\A} = 1.
\end{equation}
Furthermore, $$H(\Phi(\xi, \eta), 0) = r \nabla f(0)(\mu_a(\xi) - \nu_b(-\eta))$$ is a composition of a linear mapping and the mapping $L$ defined in \eqref{def: linking number} applied to the link ($\mu_a$, $\nu_b$). By Proposition \ref{prop: Degree for linear mappings} and the fact that the determinant of $\nabla f(0)$ is negative, we infer that $\nabla f(0)$ is sense-reversing. Therefore, Proposition \ref{prop: linking number perturbed} yields \begin{equation}
\label{eq: BV degree s=0}
\degree{H(\cdot,0)}{0}{\A}=-1.
\end{equation}
In order to obtain a contradiction with Proposition \ref{prop: Stability under homotopy}, \eqref{eq: BV degree s=1}, and \eqref{eq: BV degree s=0}, we need to show that $0 \notin H(\partial\A, s)$ for every $s \in [0,1]$. Let $\xi \in \Ss{1}$, $\eta \in \Ss{1}$. Then, using the triangle inequality together with the fact that $\nabla f(0)$ is an isometry in the first inequality, followed by the negation of \eqref{eq: BV bad links} in the second inequality, and Remarks \ref{rem: canonical link distance}, \ref{rem: distance of mu_a and mu} and \ref{rem: distance of nu_b and nu} in the third inequality, we compute
\begin{align*}
|f_s\circ\mu_a(\xi&)-f_s\circ\nu_b(-\eta)|=\\
&=|\nabla f(0)r(\mu_a(\xi)-\nu_b(\eta))+s(f(r\mu_a(\xi))-\nabla f(0)r\mu_a(\xi)) \phantom{\frac{1}{10}}\\
&\phantom{=}-s(f(r\nu_b(\eta))-\nabla f(0)r\nu_b(\eta))| \phantom{\frac{1}{10}} \\
&\geq r\abs{\mu_a(\xi)-\nu_b(\eta)} - s(g(r\mu_a(\xi))+g(r\nu_b(\eta))) \phantom{\frac{1}{10}}\\
&\geq r(\abs{\mu(\xi)-\nu(\eta)}-\abs{\mu(\xi)-\mu_a(\xi))}-\abs{\nu(\eta)-\nu_b(\eta)})-\frac{2r}{10}\\
&\geq r\left(1-\frac{1}{10}-\frac{3}{10}\right)-\frac{2r}{10} > 0.
\end{align*}
Therefore, $0 \notin H(\partial\A, s)$ and we reach a contradiction. Thus, \eqref{eq: BV bad links} is valid.

Consider the sets
\begin{equation}
\label{eq: BV Ia}
I_a = \left\{a \in \Bbar{2}\left(0,\frac{1}{10}\right): \H{1}(\{x \in \mu_a(\Ss{1}): g(rx)\geq r\delta\})< \delta \right\},   
\end{equation}
\begin{equation}
\label{eq: BV Ib}
\begin{aligned}
I_b=\left\{b \in \Bbar{2}\left(\e_1,\frac{1}{10}\right)\cap\Bbar{2}: \H{1}(\{x \in \nu_b(\Ss{1}): g(rx)\geq r\delta\}) < \delta\right\}.
\end{aligned}
\end{equation}
These sets represent the links that we consider ``good'', i.e., where the mapping $f$ is not far from $\nabla f(0)$ on a big portion of the perturbed circles $\mu_a(\Ss{1})$ and $\nu_b(\Ss{1})$ respectively. We now prove that there are enough ``good'' links. To be more precise, we claim that \eqref{eq: BV L1 diff} yields
\begin{equation}
\label{eq: BV Ia Ib estimates}
\H{2}(I_a) > \frac{1}{2}\abs{\Bbar{2}\left(0, \frac{1}{10}\right)} \text{ and } \H{2}(I_b) > \frac{1}{2} \abs{\Bbar{2}\left(\e_1,\frac{1}{10}\right)\cap\Bbar{2}}.
\end{equation}
Indeed, assume on the contrary that
\begin{equation}
\label{eq: BV good links measure}
    \H{2}(I_a) \leq \frac{1}{2}\abs{\Bbar{2}\left(0, \frac{1}{10}\right)}.
\end{equation} Denote the complement of $I_a$ in $\Bbar{2}\left(0,\frac{1}{10}\right)$ as $I_a^c$. Then by \eqref{eq: BV Ia}, we get for every $a \in I_a^c$
\begin{equation}
\label{eq: BV Ia^c estimate}
\H{1}(\{x \in \mu_a(\Ss{1}): g(rx)\geq r\delta\}) \geq \delta.
\end{equation}
By \eqref{eq: BV L1 diff}, change of variables, Area Formula, \eqref{eq: BV Ia^c estimate}, and \eqref{eq: BV good links measure},
\begin{align*}
\delta^3 &> \dashint_{\B{3}(0,4r)} \frac{g(x)}{r} = \dashint_{\B{3}(0,4)} \frac{g(rx)}{r} \\
&\geq \frac{C}{\abs{\B{3}(0,4)}} \int_{I_a^c}\int_{\mu_a(\Ss{1})} \frac{g(rx)}{r}\,d\H{1}(x)\,d\H{2}(a)\\
&\geq \frac{C}{\abs{\B{3}(0,4)}} \int_{I_a^c} \delta^2 \,d\H{2}(a)\\
&\geq \frac{C}{2\abs{\B{3}(0,4)}}\delta^2 \abs{\Bbar{2}\left(0, \frac{1}{10}\right)},
\end{align*}
which yields a contradiction for small $\delta>0$. Analogously, we can prove the estimate for $\H{2}(I_b)$ using \eqref{eq: BV L1 diff} and \eqref{eq: BV Ib}.

However, we need even ``better'' links. Thus, we consider the sets
\begin{equation}
\label{eq: BV I~a}
\Tilde{I}_a:=\left\{a \in I_a: \text{there exists }\xi \in \Ss{1} \text{ such that }g(r\mu_a(\xi))>\frac{r}{10}\right\},
\end{equation}
\begin{equation}
\label{eq: BV I~b}
\Tilde{I}_b:=\left\{b \in I_b: \text{there exists }\eta \in \Ss{1} \text{ such that }g(r\nu_b(\eta))>\frac{r}{10}\right\}
\end{equation}
and we obtain, because of \eqref{eq: BV bad links} and \eqref{eq: BV Ia Ib estimates}, either
\begin{equation}
\label{eq: BV I^~a estimate}
\H{2}(\Tilde{I}_a) > \frac{1}{4}\abs{\Bbar{2}\left(0, \frac{1}{10}\right)}
\end{equation}
or
\begin{equation}
\label{eq: BV I^~b estimate}
\H{2}(\Tilde{I}_b) > \frac{1}{4} \abs{\Bbar{2}\left(\e_1,\frac{1}{10}\right)\cap\Bbar{2}}.
\end{equation}
Using a symmetry argument, without loss of generality assume that \eqref{eq: BV I^~a estimate} holds. 

These ``better'' links are still not enough for us since we want the mapping $f$ to be of bounded variation on the links. Recall $f^\mu_a(\xi) = f(\mu_a(\xi))$ for $\xi \in \Ss{1}$. Furthermore, for a fixed $r > 0$ we denote the scaled mapping by $$\Hat{f^\mu_a}(\xi) = f(r\mu_a(\xi)).$$
Due to Corollary \ref{cor: BVL perturbed}, we have for $\H{2}$-a.e. $a \in \Bbar{2}\left(0,\frac{1}{10}\right)$ that $$\Hat{f^\mu_a} \in BV(\Ss{1}, \R^3).$$ Hence, $$\Hat{I}_a := \left\{a \in \Tilde{I}_a: f \in BV(r\mu_a(\Ss{1}), \R^3) \right\}$$ satisfies the same estimate as $\Tilde{I}_a$, namely 
\begin{equation}
\label{eq: BV Î is big enough}
\H{2}(\Hat{I}_a) > \frac{1}{4}\abs{\Bbar{2}\left(0, \frac{1}{10}\right)}.
\end{equation}

Set $$A:= \{x \in \B{3}(0, 4r): g(x) > \delta r\}.$$ This is the small set where, as we show later on, the mapping $f$ is highly oscillating. In fact, because of Chebyshev's inequality and \eqref{eq: BV L1 diff} we obtain
\begin{equation}
\label{eq: BV estimate of |A|}
|A| \leq \int_{\B{3}(0,4r)} \frac{g(x)}{\delta r}\,dx \leq \delta^{2}|\B{3}(0,4r)|.
\end{equation}
Thus, \eqref{eq: BV condition 1} is satisfied. Denote $\Hat{A^\mu_a} \subset \Ss{1}$ the part of the circle $\Ss{1}$ that is mapped by $r\mu_a$ into $A$.

Now we prove that $f$ is highly oscillating on $A$. Let $a \in \Hat{I}_a$. Choose $x = \mu_a(\xi)$ from \eqref{eq: BV I~a} and, thanks to \eqref{eq: BV Ia} and the fact that $g$ is continuous, we find $\zeta \in \Ss{1}$ (we denote $y = \mu_a(\zeta)$) such that $g(ry) = r \delta$ and $\abs{x-y} < 2\delta$. From the definition of the variation on line segments and the fact that we can consider the variation only on $A$, we observe that
\begin{equation}
\label{eq: BV difference leq variation}
    |\Hat{f^\mu_a}(\xi) - \Hat{f^\mu_a}(\zeta)| \leq \abs{D\Hat{f^\mu_a}}\left(\Hat{A^\mu_a}\right).
\end{equation}
Furthermore, by using the triangle inequality for these $x, y$ it follows that
\begin{align*}
|f(rx) - f(ry)| &= |(f(rx) - r\nabla f(0) x) - (f(ry) - r \nabla f(0)y) + r\nabla f(0) (x-y)|\\
&\geq g(rx) - g(ry) - r|\nabla f(0)(x-y)|\phantom{\frac{1}{10}}\\
&\geq \frac{r}{10} - \delta r - r|x-y|\\
&\geq r \left(\frac{1}{10}-\delta - 2\delta\right).
\end{align*}
Since $\delta < \frac{1}{100}$, we get
\begin{equation}
\label{eq: BV concrete difference}
|\Hat{f^\mu_a}(\xi) - \Hat{f^\mu_a}(\zeta)| = \abs{f(rx) - f(ry)} \geq \frac{r}{20}.
\end{equation}
Now we are ready for a final computation. First, we integrate over $\Hat{I}_a$ and combine \eqref{eq: BV difference leq variation} and \eqref{eq: BV concrete difference} and then use Lemma \ref{lem: BV final computation} to obtain
\begin{align*}
    \int_{\Hat{I}_a} \frac{r}{20} &\leq \int_{\Bbar{2}\left(0, \frac{1}{10}\right)} |D(\Hat{f^\mu_a})|(\Hat{A^\mu_a}) \, d\H{2}(a)\\
    &\leq C r^{-2} |Df|(A).
\end{align*}
From the mutual singularity of the measures $D^af$ and $D^sf$ and the estimate \eqref{eq: BV singular part}, we have
\begin{equation}
\label{eq: BV final computation}
\begin{aligned}
\int_{\Hat{I}_a} \frac{r}{20} &\leq Cr^{-2}\left(\int_A |\nabla f(x)| \, dx + |D^sf|(A)\right) \phantom{\int}\\
&< Cr^{-2}\left(\int_A |\nabla f(x)| \, dx + \Hat{C} \delta |\B{3}(0, 4r)|\right).
\end{aligned}
\end{equation}
One can observe that our constant $C$ is absolute and, thus, we can plug its value into \eqref{eq: BV HatC estimate}.

Altogether, \eqref{eq: BV HatC estimate} and \eqref{eq: BV final computation} yield ($C$ is a different constant than the one in \eqref{eq: BV final computation})
$$r^3 \leq C \int_A |\nabla f(x)| \, dx.$$
We set $C_1 = \max\{1, C\}$ and obtain \eqref{eq: BV condition 2}. This concludes the proof.
\end{proof}

\section{Weak-* limits of $BV(\Omega, \R^n)$ homeomorphisms}
\label{chapter: BV*}

\subsection{The sign of the Jacobian - new tools}
\label{section: BV* Jacobian new tools}

The first lemma is a simple observation that lies at the core of a more interesting inductive process described in Lemma \ref{lem: Mountain Range}.

\begin{lemma}[Two Hills Lemma]
\label{lem: Two Hills}
Let $a, b, c, d \in \R$ satisfy $a \geq d$ and $b \leq c$. Then $$|a-b| + |b-c| + |c-d| \geq |a-d| + |b-c|.$$
\end{lemma}

\begin{proof}
First, we notice that the inequality in question is equivalent to
$$|a-b| + |c-d| \geq |a-d|.$$
Due to the fact that $c-b \geq 0$ and $a-d \geq 0$, we get
$$|a-b| + |c-d| \geq (a-b) + (c-d) = a + (c-b) - d \geq a-d = |a-d|.$$
\end{proof}

\begin{lemma}[Mountain Range Lemma]
\label{lem: Mountain Range}
Let $g \in BV([-2,2])$. For a fixed $j \in \N$, consider the set of indices
\begin{equation}
\label{eq: BV* I^j definition}
\begin{aligned}
I^j = &\{i \in \{-2^{j+1}+1, \ldots, 0, \ldots, 2^{j+1}\}: \text{there exist }\\ &(i-1) 2^{-j} < \eta_{i, j}^1 < \xi_{i, j} < \eta^2_{i,j} < i 2^{-j} \text{ such that }\\
&g(\xi_{i, j}) - \max \{g(\eta^1_{i,j}), g(\eta^2_{i,j})\} \geq 2^{-j-3}\}.
\end{aligned}
\end{equation}
Let $n \geq 2^{30}$
be such that there exists a set $J = \{16 \leq j_0 < \ldots < j_N \leq n\}$ with 
\begin{equation}
\label{eq: BV* choice of N}
N \geq 2^{-30}n    
\end{equation}
such that for every $j \in J$ we have the estimate 
\begin{equation}
\label{eq: BV* I^j estimate}
|I^j| \geq 2^{j-15}.
\end{equation} Then $$|Dg|([-2,2]) \geq 2^{-50} n.$$
\end{lemma}

\begin{idea}
The condition \eqref{eq: BV* I^j estimate} guarantees that there is a large number of ``spikes'' $(\eta^1, \xi, \eta^2)$ on every ``scale'' $j \in J$. The conclusion states that the function $g$ is wildly oscillating.

We do an induction on the number dyadic scales. We start with the first scale $j_0 \in J$. Every $i \in I^{j_0}$ represents an open interval $((i-1)2^{-j_0}, i2^{-j_0})$ containing a spike. Hence, we create a ``mountain range'' consisting of $2^{j_0-16} < |I^{j_0}|$ spikes $(\eta^1, \xi, \eta^2)$ satisfying $$g(\xi) - \max \{g(\eta^1), g(\eta^2)\} \geq 2^{-j_0-3}.$$ In this manner we can construct a partition $D^1_0$ of the interval $[-2, 2]$ such that $$V(g, D^1_0) \geq M_0 2^{-j_0-3} = 2^{-19}.$$

Next, suppose we have for the $k$-th scale $j_k$ a partition $D^k_0$ satisfying
\begin{equation}
\label{eq: BV* V(g, D^k_0) estimate}
V(g, D^k_0) \geq k 2^{-19}.  
\end{equation}
We denote $M_k = 2^{-j_k-16}$ and inductively construct partitions $D^k_l$ for $l \in \{1, \ldots, M_k\}$. Suppose we already constructed $D^k_{l}$. To construct $D^k_{l+1}$, we add a triplet of points $(\eta^1, \xi, \eta^2)$ such that $$g(\xi) - \max \{g(\eta^1), g(\eta^2)\} \geq 2^{-j_k-3}.$$ It means that the partition $D^k_{l+1}$ witnesses one more place where $g$ significantly oscillates, i.e., $$V(g, D^k_{l+1}) \geq V(g, D^k_{l}) + 2^{-j_k-3}.$$ Since, on every scale $j_k$ the partitions $D^k_l$ successively witness $M_k = 2^{j_k-16}$ new spikes, we find that every scale adds a constant $2^{-19}$ to the variation of $g$. This is enough to conclude the proof of \eqref{eq: BV* V(g, D^k_0) estimate}. Then we end the proof by setting $k := N \geq 2^{-30}n$.
\end{idea}

\begin{proof}[Proof of Lemma \ref{lem: Mountain Range}]
Fix $n \geq 2^{30}$. For $j_k \in J$, let us denote $M_k = 2^{j_k - 16}$. Inductively, for all $k \in \{1, \ldots, N\}$ and $l \in \{1, \ldots, M_k\}$ we construct partitions $D^k_l = \{t_r\}_{r=1}^{N^k_l}$ that satisfy the auxiliary \emph{hill-valley condition}
\begin{equation}
\label{eq: BV* hill-valley inductive estimate}
    g(t_{2m}) \geq \max\{g(t_{2m - 1}), g(t_{2m + 1})\}
\end{equation} for every possible $m \in \N$ and, moreover, 
\begin{equation}
\label{eq: BV* partitions inductive estimate}
V(g, D^k_l) \geq k 2^{-19} + l 2^{-j_k-3}.
\end{equation}
Once we know \eqref{eq: BV* partitions inductive estimate} for every suitable $k$ and $l$, we can use \eqref{eq: BV* choice of N} and plug in $k = N$ and $l=0$ to obtain 
$$|Dg|([-2,2]) \geq V(g, D^N_0) \geq N 2^{-19} \geq 2^{-50} n,$$ which concludes the whole proof.

From now on, we refer to points $t_{2m}$ as \emph{hills} and to points $t_{2m + 1}$ as \emph{valleys}.
Let us start with $D^0_0 = \emptyset$. Consider $j_0 \in J$ and, thanks to \eqref{eq: BV* I^j estimate}, let $\Tilde{I^0}$ be a subset of $I^0$ such that $|\Tilde{I^0}| = M_0$. Let us write $\Tilde{I^0} = \{i_1 < \ldots < i_{M_0}\}$. Then we define $D^0_1$ to be the partition consisting of two valleys $t_1 = \eta^1$, $t_3 = \eta^2$ together with the hill $t_2 = \xi$. Clearly, the condition \eqref{eq: BV* hill-valley inductive estimate} is satisfied and, since $$V(g, D^0_1) \geq |g(\xi) - g(\eta^1)| \geq 2^{-j_0-3},$$ we also get \eqref{eq: BV* partitions inductive estimate}.

Now, suppose that we have suitable $D^0_l$ and we want to construct $D^0_{l+1}$ by adding another ``spike'' $(\eta^1, \xi, \eta^2)$. In order to satisfy \eqref{eq: BV* hill-valley inductive estimate}, we need to compare the last point $t^0_l$ of the partition $D^0_l$ with our new point $\eta^1$. Thus, we keep only the point $t^0_l$ if $g(t^0_l) < g(\eta^1)$ and keep $\eta^1$ otherwise. Then we append $\xi$ and $\eta^2$ to finish the construction of $D^0_{l+1}$. See Figure \ref{fig: BV* D^0_2 and spike}. We made sure to satisfy \eqref{eq: BV* hill-valley inductive estimate} and, moreover, it is easy to check $$V(g, D^0_{l+1}) \geq V(g, D^0_l) + |g(\xi) - g(\eta^2)| \geq (l+1) 2^{-j_0-3},$$ which means that \eqref{eq: BV* partitions inductive estimate} is satisfied as well.

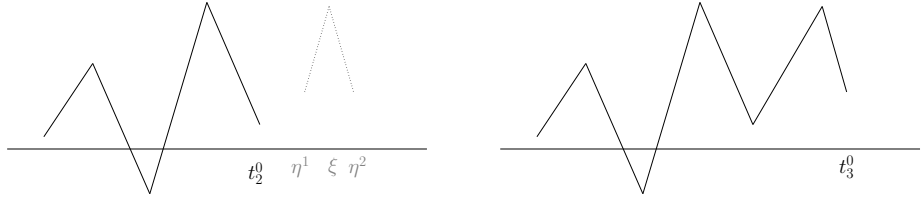
\begin{figure}[H]
\centering
\resizebox{1\textwidth}{!}{%
\begin{circuitikz}[yscale = 1]
\tikzstyle{every node}=[font=\fontsize{10.2pt}{13.3pt}\selectfont]
\draw [short] (2.625,7.75) -- (15.5,7.75);
\draw [short] (3.75,8.125) -- (5.25,10.375);
\draw [short] (5.25,10.375) -- (7,6.375);
\draw [short] (7,6.375) -- (8.75,12.25);
\draw [short] (8.75,12.25) -- (10.375,8.5);
\draw [dotted] (11.75,9.5) -- (12.5,12.125);
\draw [dotted] (12.5,12.125) -- (13.25,9.5);
\node [font=\fontsize{18.2pt}{21.3pt}\selectfont, text opacity=0.5, , inner xsep=0.080cm, inner ysep=0.085cm, rounded corners=0.020cm] at (11.625,7.2) {$\eta^1$};
\node [font=\fontsize{18.2pt}{21.3pt}\selectfont, text opacity=0.5, , inner xsep=0.080cm, inner ysep=0.085cm, rounded corners=0.020cm] at (13.375,7.2) {$\eta^2$};
\node [font=\fontsize{18.2pt}{21.3pt}\selectfont, text opacity=0.5, , inner xsep=0.080cm, inner ysep=0.085cm, rounded corners=0.000cm] at (12.625,7.2) {$\xi$};
\node [font=\fontsize{18.2pt}{21.3pt}\selectfont, inner xsep=0.080cm, inner ysep=0.085cm, rounded corners=0.020cm] at (10.25,7) {$t^0_2$};
\end{circuitikz}
\hspace{2 cm}
\begin{circuitikz}[yscale = 1]
\tikzstyle{every node}=[font=\fontsize{10.5pt}{13.3pt}\selectfont]
\draw [short] (2.625,7.75) -- (15.5,7.75);
\draw [short] (3.75,8.125) -- (5.25,10.375);
\draw [short] (5.25,10.375) -- (7,6.375);
\draw [short] (7,6.375) -- (8.75,12.25);
\draw [short] (8.75,12.25) -- (10.375,8.5);
\draw [thick] (12.5,12.125) -- (13.25,9.5);
\draw [thick] (10.375,8.5) -- (12.5,12.125);
\node [font=\fontsize{18.2pt}{21.3pt}\selectfont, inner xsep=0.080cm, inner ysep=0.085cm, rounded corners=0.020cm] at (13.25,7.2) {$t^0_3$};
\end{circuitikz}
}%
\caption{On the left, we have $D^0_2$ with its last point $t^0_2$ and a new dotted spike $(\eta^1, \xi, \eta^2)$. On the right, we constructed $D^0_3$ by deleting $\eta^1$ (since $g(\eta^1) > g(t^0_2)$) and adding $\xi, \eta^2$. Finally, $\eta^2$ as the last point of $D^0_3$ becomes $t^0_3$.}
\label{fig: BV* D^0_2 and spike}
\end{figure}

In this way, we can define $D^0_{M_0}$. Then we set $D^1_0 = D^0_{M_0}$ and we observe that
$$V(g, D^1_0) = V(g, D^0_{M_0}) \geq M_0 2^{-j_0-3} = 2^{j_0 - 16} 2^{-j_0-3} = 2^{-19}$$
and, hence, $D^1_0$ satisfies \eqref{eq: BV* partitions inductive estimate}. In the same manner, if we have $D^k_{M_k}$, we set $D^{k+1}_0 = D^k_{M_k}$ and we easily check that both \eqref{eq: BV* hill-valley inductive estimate} and \eqref{eq: BV* partitions inductive estimate} are true for $D^{k+1}_0$. Therefore, it remains to inductively construct $D^k_{l+1}$ from $D^k_l$.

First, we need to choose ``good'' $i \in I^{j_k}$ so that we add $2^{-19}$ to the total variation over the entire $k$-th scale. Therefore, we define $$\hat{I}^k = \{i \in I^{j_k}: \text{there is no hill of } D^k_0 \text{ in }I^{j_k}_i\}.$$
Since for every $\kappa < k$ we added in total $M_\kappa = 2^{j_\kappa-16}$ hills, we obtain thanks to \eqref{eq: BV* I^j estimate}
\begin{align*}
|\hat{I}^k| &\geq |I^{j_k}| - \sum_{\kappa = 0}^{k-1} M_\kappa\\
&\geq 2^{j_k - 15} - 2^{-16} \sum_{\kappa=0}^{k-1} 2^{j_\kappa}\\
&\geq 2^{-15} (2^{j_k} - 2^{-1}\sum_{j=0}^{j_k - 1}2^j)\\
&\geq 2^{-15} (2^{j_k} - 2^{j_k - 1}) = 2^{j_k - 16} = M_k.
\end{align*}
Hence, we can define $\Tilde{I}^k = \{i_1 < i_2 < \ldots < i_{M_k} \} \subset \hat{I}^k$. This is the set of indices $i_l$ that we consider in the inductive process.

Now, let $l \in \{0, \ldots, M_k - 1\}$ and consider $D^k_{l}$. We want $D^k_{l+1}$ to see the oscillation on the new spike $(\eta^1, \xi, \eta^2)$. However, we need to know that the oscillation over the spike $(\eta^1, \xi, \eta^2)$ is not already accounted for in $V(g, D^k_l)$. This is the reason why we impose the condition \eqref{eq: BV* hill-valley inductive estimate}. To be more precise, we construct $D^k_{l+1}$ from $D^k_{l}$ by possibly discarding some points and adding a certain subset of $\{\eta^1, \xi, \eta^2\}$ in order so that $D^k_{l+1}$ satisfies \eqref{eq: BV* hill-valley inductive estimate}. Thus, we distinguish 2 possible cases of the relative position of the new hill $\xi$ with respect to the already existing hills and valleys of $D^k_l$. In the \textbf{First case}, we assume that $\xi$ coincides with some valley $t_\iota$ of $D^k_l$. Then we simply add $\eta^1$ and $\eta^2$ into the partition (see Figure \ref{fig: BV* First case 1}). In the \textbf{Second case}, we suppose that $\xi$ is not equal to any valley of $D^k_l$. Then, since $i_l \in \hat{I}^k$, we infer that $\xi$ is not equal to any hill of $D^k_l$. Thus, $\xi \in (t_\iota, t_{\iota+1})$ where $t_\iota, t_{\iota + 1} \in D^k_l$ are two consecutive points of the partition. Thanks to \eqref{eq: BV* hill-valley inductive estimate}, we see that one of the points $t_\iota, t_{\iota + 1}$ is a hill and the other is a valley. Without loss of generality, assume that $t_\iota$ is a hill. We start constructing $D^k_{l+1}$ by adding $\xi$ to $D^k_l$. Next, we compare the value of $g(\eta^1)$ with $g(t_\iota)$. If $g(\eta^1) \leq g(t_\iota)$, we add $\eta^1$. If $g(\eta^1) > g(t_\iota)$, we add $t_\iota$ -- in this case, $t_\iota$ is a hill and a valley of $D^k_{l+1}$ at the same time. Finally, we compare $g(\eta^2)$ with $g(t_{\iota+1})$ and only keep the point out of the corresponding pair where the functional value of $g$ is smaller. We sketch various situations in Figures \ref{fig: BV* Second case 1}, \ref{fig: BV* Second case 2} and \ref{fig: BV* Second case 3}.

In this way, we obtain $D^k_{l+1}$ that satisfies \eqref{eq: BV* hill-valley inductive estimate}. Furthermore, since $D^k_{l+1}$ sees one new spike $(\eta^1, \xi, \eta^2)$ we get, due to Lemma \ref{lem: Two Hills}, \eqref{eq: BV* partitions inductive estimate} and \eqref{eq: BV* I^j definition}, in both cases
\begin{align*}
V(g, D^k_{l+1}) &\geq V(g, D^k_l) + g(\xi) - \max \{g(\eta^1), g(\eta^2)\}\\ &\geq k2^{-19} + l 2^{-j_k-3} + 2 ^{-j_k - 3}\\
&\geq k2^{-19} + (l+1) 2^{-j_k-3}.
\end{align*} Thus, \eqref{eq: BV* partitions inductive estimate} holds for $D^k_{l+1}$ and we conclude the inductive process. This concludes the whole proof.
\end{proof}

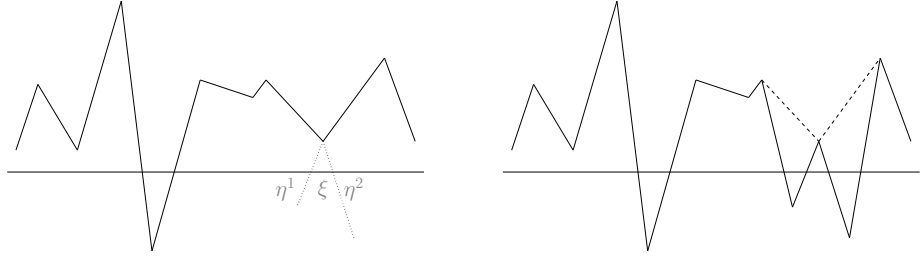
\begin{figure}[H]
\centering
\resizebox{1\textwidth}{!}{%
\begin{circuitikz}[yscale = 1]
\tikzstyle{every node}=[font=\fontsize{10.2pt}{13.3pt}\selectfont]
\draw [short] (3.5,7.5) -- (15.375,7.5);
\draw [short] (4.375,10) -- (5.5,8.125);
\draw [short] (5.5,8.125) -- (6.75,12.375);
\draw [short] (6.75,12.375) -- (7.625,5.25);
\draw [short] (7.625,5.25) -- (9,10.125);
\draw [short] (9,10.125) -- (10.5,9.625);
\draw [short] (10.5,9.625) -- (10.875,10.125);
\draw [short] (10.875,10.125) -- (12.5,8.375);
\draw [short] (12.5,8.375) -- (14.25,10.75);
\draw [short] (14.25,10.75) -- (15.125,8.375);
\draw [short] (4.375,10) -- (3.75,8.125);
\draw [dotted] (12.5,8.375) -- (11.75,6.5);
\draw [dotted] (12.5,8.375) -- (13.375,5.625);
\node [font=\fontsize{18.2pt}{21.3pt}\selectfont, text opacity=0.5, , fill={rgb,255:red,255; green,255; blue,255}, fill opacity=0.5, text opacity=0.5, inner xsep=0.080cm, inner ysep=0.085cm, rounded corners=0.020cm] at (11.425,7) {$\eta^1$};
\node [font=\fontsize{18.2pt}{21.3pt}\selectfont, text opacity=0.5, , fill={rgb,255:red,255; green,255; blue,255}, fill opacity=0.5, text opacity=0.5, inner xsep=0.080cm, inner ysep=0.085cm, rounded corners=0.020cm] at (13.375,7) {$\eta^2$};
\node [font=\fontsize{18.2pt}{21.3pt}\selectfont, text opacity=0.5, , fill={rgb,255:red,255; green,255; blue,255}, fill opacity=0.5, text opacity=0.5, inner xsep=0.080cm, inner ysep=0.085cm, rounded corners=0.000cm] at (12.5,7) {$\xi$};
\end{circuitikz}
\hspace{2 cm}
\begin{circuitikz}[yscale = 1]
\tikzstyle{every node}=[font=\fontsize{10.2pt}{13.3pt}\selectfont]
\draw [short] (3.5,7.5) -- (15.375,7.5);
\draw [short] (4.375,10) -- (5.5,8.125);
\draw [short] (5.5,8.125) -- (6.75,12.375);
\draw [short] (6.75,12.375) -- (7.625,5.25);
\draw [short] (7.625,5.25) -- (9,10.125);
\draw [short] (9,10.125) -- (10.5,9.625);
\draw [short] (10.5,9.625) -- (10.875,10.125);
\draw [dashed] (10.875,10.125) -- (12.5,8.375);
\draw [dashed] (12.5,8.375) -- (14.25,10.75);
\draw [short] (14.25,10.75) -- (15.125,8.375);
\draw [short] (4.375,10) -- (3.75,8.125);
\draw [thick] (10.875,10.125) -- (11.75,6.5);
\draw [thick] (11.75,6.5) -- (12.5,8.375);
\draw [thick] (12.5,8.375) -- (13.375,5.625);
\draw [thick] (13.375,5.625) -- (14.25,10.75);
\end{circuitikz}
}%
\caption{\textbf{First case:} $\xi$ is equal to some valley of $D^k_l$. On the left, we see $D^k_l$ with the dotted new spike $(\eta^1, \xi, \eta^2)$. Then $\xi$ becomes the new hill and $\eta^1$ and $\eta^2$ become new valleys in  $D^k_{l+1}$, which is depicted on the right.}
\label{fig: BV* First case 1}
\end{figure}

\begin{figure}[H]
\centering
\resizebox{1\textwidth}{!}{%
\begin{circuitikz}[yscale = 1]
\tikzstyle{every node}=[font=\fontsize{18.2pt}{21.3pt}\selectfont]
\draw [short] (3.5,7.5) -- (15.375,7.5);
\draw [short] (4.375,10) -- (5.5,8.125);
\draw [short] (5.5,8.125) -- (6.75,12.375);
\draw [short] (6.75,12.375) -- (7.625,5.25);
\draw [short] (7.625,5.25) -- (9,10.125);
\draw [short] (9,10.125) -- (10.5,9.625);
\draw [short] (10.5,9.625) -- (10.875,10.125);
\draw [short] (10.875,10.125) -- (12.5,8.375);
\draw [short] (12.5,8.375) -- (14.25,10.75);
\draw [short] (14.25,10.75) -- (15.125,8.375);
\draw [short] (4.375,10) -- (3.75,8.125);
\draw [dotted] (12.75,10.625) -- (13.375,12.25);
\draw [dotted] (13.375,12.25) -- (13.875,11);
\node [font=\fontsize{18.2pt}{21.3pt}\selectfont, text opacity=0.5, , fill={rgb,255:red,255; green,255; blue,255}, fill opacity=0.5, text opacity=0.5, inner xsep=0.080cm, inner ysep=0.085cm, rounded corners=0.020cm] at (12.625,7) {$\eta^1$};
\node [font=\fontsize{18.2pt}{21.3pt}\selectfont, text opacity=0.5, , fill={rgb,255:red,255; green,255; blue,255}, fill opacity=0.5, text opacity=0.5, inner xsep=0.080cm, inner ysep=0.085cm, rounded corners=0.020cm] at (13.875,7) {$\eta^2$};
\node [font=\fontsize{18.2pt}{21.3pt}\selectfont, text opacity=0.5, , fill={rgb,255:red,255; green,255; blue,255}, fill opacity=0.5, text opacity=0.5, inner xsep=0.080cm, inner ysep=0.085cm, rounded corners=0.000cm] at (13.25,7) {$\xi$};
\end{circuitikz}
\hspace{2 cm}
\begin{circuitikz}[yscale = 1]
\tikzstyle{every node}=[font=\fontsize{10.2pt}{13.3pt}\selectfont]
\draw [short] (3.5,7.5) -- (15.375,7.5);
\draw [short] (4.375,10) -- (5.5,8.125);
\draw [short] (5.5,8.125) -- (6.75,12.375);
\draw [short] (6.75,12.375) -- (7.625,5.25);
\draw [short] (7.625,5.25) -- (9,10.125);
\draw [short] (9,10.125) -- (10.5,9.625);
\draw [short] (10.5,9.625) -- (10.875,10.125);
\draw [short] (10.875,10.125) -- (12.5,8.375);
\draw [dashed] (12.5,8.375) -- (14.25,10.75);
\draw [short] (14.25,10.75) -- (15.125,8.375);
\draw [short] (4.375,10) -- (3.75,8.125);
\draw [thick] (12.5,8.375) -- (13.375,12.25);
\draw [thick] (13.375,12.25) -- (14.25,10.75);
\node [font=\fontsize{18.2pt}{21.3pt}\selectfont, fill={rgb,255:red,255; green,255; blue,255}, fill opacity=1, text opacity=1, inner xsep=0.080cm, inner ysep=0.085cm, rounded corners=0.000cm] at (14.375,7) {$t^*$};
\end{circuitikz}
}%
\caption{\textbf{Second case:} $\xi$ is between a valley and a hill of $D^k_l$. On the left, we have $D^k_l$ with the new dotted spike $(\eta^1, \xi, \eta^2)$. Then $\xi$ becomes a new hill in $D^k_{l+1}$. In this situation, both $\eta^1$ and $\eta^2$ are discarded because the surrounding points from $D^k_l$ have smaller functional value. On the right, we have $D^k_{l+1}$. Note that, in our construction, the point $t^*$ is a hill and a valley at the same time in order that $D^k_{l+1}$ satisfies \eqref{eq: BV* hill-valley inductive estimate}.}
\label{fig: BV* Second case 1}
\end{figure}
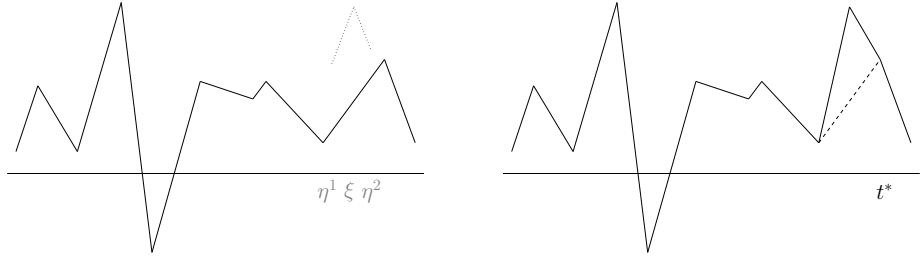

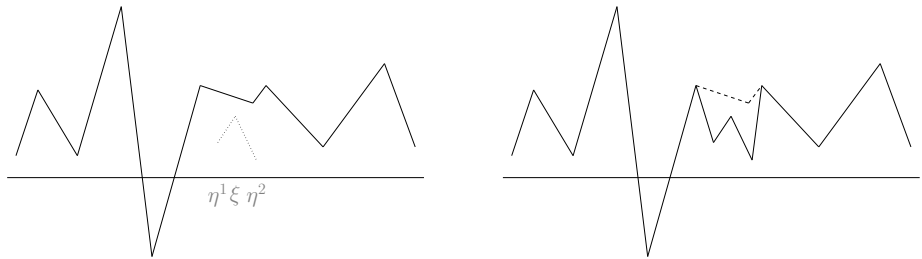
\begin{figure}[H]
\centering
\resizebox{1\textwidth}{!}{%
\begin{circuitikz}[yscale = 1]
\tikzstyle{every node}=[font=\fontsize{10.2pt}{13.3pt}\selectfont]
\draw [short] (3.5,7.5) -- (15.375,7.5);
\draw [short] (4.375,10) -- (5.5,8.125);
\draw [short] (5.5,8.125) -- (6.75,12.375);
\draw [short] (6.75,12.375) -- (7.625,5.25);
\draw [short] (7.625,5.25) -- (9,10.125);
\draw [short] (9,10.125) -- (10.5,9.625);
\draw [short] (10.5,9.625) -- (10.875,10.125);
\draw [short] (10.875,10.125) -- (12.5,8.375);
\draw [short] (12.5,8.375) -- (14.25,10.75);
\draw [short] (14.25,10.75) -- (15.125,8.375);
\draw [short] (4.375,10) -- (3.75,8.125);
\draw [dotted] (9.5,8.5) -- (10,9.25);
\draw [dotted] (10,9.25) -- (10.6,8);
\node [font=\fontsize{18.2pt}{21.3pt}\selectfont, text opacity=0.5, , fill={rgb,255:red,255; green,255; blue,255}, fill opacity=0.5, text opacity=0.5, inner xsep=0.080cm, inner ysep=0.085cm, rounded corners=0.020cm] at (9.5,7) {$\eta^1$};
\node [font=\fontsize{18.2pt}{21.3pt}\selectfont, text opacity=0.5, , fill={rgb,255:red,255; green,255; blue,255}, fill opacity=0.5, text opacity=0.5, inner xsep=0.080cm, inner ysep=0.085cm, rounded corners=0.020cm] at (10.6,7) {$\eta^2$};
\node [font=\fontsize{18.2pt}{21.3pt}\selectfont, text opacity=0.5, , fill={rgb,255:red,255; green,255; blue,255}, fill opacity=0.5, text opacity=0.5, inner xsep=0.080cm, inner ysep=0.085cm, rounded corners=0.000cm] at (10,7) {$\xi$};
\end{circuitikz}
\hspace{2 cm}
\begin{circuitikz}[yscale = 1]
\tikzstyle{every node}=[font=\fontsize{10.2pt}{13.3pt}\selectfont]
\draw [short] (3.5,7.5) -- (15.375,7.5);
\draw [short] (4.375,10) -- (5.5,8.125);
\draw [short] (5.5,8.125) -- (6.75,12.375);
\draw [short] (6.75,12.375) -- (7.625,5.25);
\draw [short] (7.625,5.25) -- (9,10.125);
\draw [dashed] (9,10.125) -- (10.5,9.625);
\draw [dashed] (10.5,9.625) -- (10.875,10.125);
\draw [short] (10.875,10.125) -- (12.5,8.375);
\draw [short] (12.5,8.375) -- (14.25,10.75);
\draw [short] (14.25,10.75) -- (15.125,8.375);
\draw [short] (4.375,10) -- (3.75,8.125);
\draw[thick] (9,10.125) -- (9.5,8.5);
\draw[thick] (9.5,8.5) -- (10,9.25);
\draw[thick] (10,9.25) -- (10.6,8);
\draw[thick] (10.6,8) -- (10.875,10.125);
\end{circuitikz}
}%
\caption{\textbf{Second case:} $\xi$ is between a valley and a hill of $D^k_l$. On the left, we have $D^k_l$ with the new dotted spike $(\eta^1, \xi, \eta^2)$. Then $\xi$ becomes a new hill in $D^k_{l+1}$. In this situation, both $\eta^1$ and $\eta^2$ are added because the surrounding points from $D^k_l$ have bigger functional values. On the right, we have $D^k_{l+1}$.}
\label{fig: BV* Second case 2}
\end{figure}

\begin{figure}[H]
\centering
\resizebox{1\textwidth}{!}{%
\begin{circuitikz}[yscale = 1]
\tikzstyle{every node}=[font=\fontsize{18.2pt}{21.3pt}\selectfont]
\draw [short] (3.5,7.5) -- (15.375,7.5);
\draw [short] (4.375,10) -- (5.5,8.125);
\draw [short] (5.5,8.125) -- (6.75,12.375);
\draw [short] (6.75,12.375) -- (7.625,5.25);
\draw [short] (7.625,5.25) -- (9,10.125);
\draw [short] (9,10.125) -- (10.5,9.625);
\draw [short] (10.5,9.625) -- (10.875,10.125);
\draw [short] (10.875,10.125) -- (12.5,8.375);
\draw [short] (12.5,8.375) -- (14.25,10.75);
\draw [short] (14.25,10.75) -- (15.125,8.375);
\draw [short] (4.375,10) -- (3.75,8.125);
\draw [dotted] (11.5,9) -- (13,10);
\draw [dotted] (13,10) -- (13.5,9);
\node [font=\fontsize{18.2pt}{21.3pt}\selectfont, fill={rgb,255:red,255; green,255; blue,255}, fill opacity=1, text opacity=1, inner xsep=0.080cm, inner ysep=0.085cm, rounded corners=0.000cm] at (12.5,7) {$t_\iota$};
\node [font=\fontsize{18.2pt}{21.3pt}\selectfont, fill={rgb,255:red,255; green,255; blue,255}, fill opacity=1, text opacity=1, inner xsep=0.080cm, inner ysep=0.085cm, rounded corners=0.000cm] at (14.3,7) {$t_{\iota+1}$};
\node [font=\fontsize{18.2pt}{21.3pt}\selectfont, text opacity=0.5, , fill={rgb,255:red,255; green,255; blue,255}, fill opacity=0.5, text opacity=0.5, inner xsep=0.080cm, inner ysep=0.085cm, rounded corners=0.020cm] at (11.5,7) {$\eta^1$};
\node [font=\fontsize{18.2pt}{21.3pt}\selectfont, text opacity=0.5, , fill={rgb,255:red,255; green,255; blue,255}, fill opacity=0.5, text opacity=0.5, inner xsep=0.080cm, inner ysep=0.085cm, rounded corners=0.020cm] at (13.5,7) {$\eta^2$};
\node [font=\fontsize{18.2pt}{21.3pt}\selectfont, text opacity=0.5, , fill={rgb,255:red,255; green,255; blue,255}, fill opacity=0.5, text opacity=0.5, inner xsep=0.080cm, inner ysep=0.085cm, rounded corners=0.000cm] at (13,7) {$\xi$};
\end{circuitikz}
\hspace{2 cm}
\begin{circuitikz}[yscale = 1]
\tikzstyle{every node}=[font=\fontsize{10.2pt}{13.3pt}\selectfont]
\draw [short] (3.5,7.5) -- (15.375,7.5);
\draw [short] (4.375,10) -- (5.5,8.125);
\draw [short] (5.5,8.125) -- (6.75,12.375);
\draw [short] (6.75,12.375) -- (7.625,5.25);
\draw [short] (7.625,5.25) -- (9,10.125);
\draw [short] (9,10.125) -- (10.5,9.625);
\draw [short] (10.5,9.625) -- (10.875,10.125);
\draw [short] (10.875,10.125) -- (12.5,8.375);
\draw [dashed] (12.5,8.375) -- (14.25,10.75);
\draw [short] (14.25,10.75) -- (15.125,8.375);
\draw [short] (4.375,10) -- (3.75,8.125);
\draw [thick] (12.5,8.375) -- (13,10);
\draw [thick] (13,10) -- (13.5,9);
\draw [thick] (13.5,9) -- (14.25,10.75);
\end{circuitikz}
}%
\caption{\textbf{Second case:} $\xi$ is between a valley and a hill of $D^k_l$. On the left, we have $D^k_l$ with the new dotted spike $(\eta^1, \xi, \eta^2)$. Then $\xi$ becomes a new hill in $D^k_{l+1}$. In this situation, $\eta^1$ is discarded because the valley $t_\iota$ satisfies $g(t_\iota) < g(\eta^1)$. However, we add $\eta^2$ because $g(\eta^2) < g(t_{\iota+1})$. On the right, we have $D^k_{l+1}$.}
\label{fig: BV* Second case 3}
\end{figure}
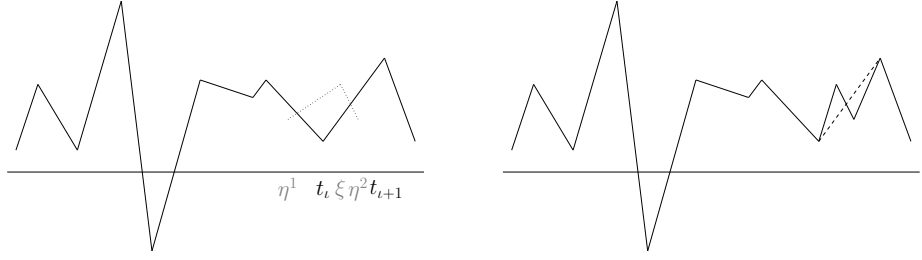

The following lemma is a bridge between the setting of Theorem \ref{thm: BV limit jacobian} in $\R^3$ and the one-dimensional Lemma \ref{lem: Mountain Range}. Recall that for a cube $Q \subset \R^3$ we denote $h(Q) \subset \R^2$ its upper face and for $(x^*, y^*) \in \R^2$ we denote $P_{x^*, y^*} \subset \R^3$ the line satisfying for every $(x, y, z) \in P_{x^*, y^*}$ the equalities $x = x^*$ and $y = y^*$.

\begin{lemma}
\label{lem: BV limit Dgk unbounded}
Let $Q_0 = Q(0, 4) \subset \R^3$, $g_k \in BV(Q_0)$, and $\delta \in (0, 10^{-10}).$ Suppose that there exists an increasing sequence of natural numbers $\mathcal{N} =\{n_q\}_{q = 1}^\infty$ such that
\begin{equation}
\label{eq: BV* N cap 1...n_q estimate}
\text{for infinitely many $q \in \N$ we have } |\mathcal{N} \cap \{1, \ldots, n_q\}| \geq \frac{n_q}{100}
\end{equation}
and for every $n = n_q \in \mathcal{N}$ there exists $k_n \in \N$ such that for every $k \geq k_n$ the following claim holds. Consider $\{Q^n_i\}_{i=1}^{8^{n+1}}$ a dyadic partition of $Q_0$ into cubes of side length $2^{1-n}$. There exists $\mathcal{Q}^k \subset \{Q^n_i\}_{i=1}^{8^{n+1}}$ satisfying $|\mathcal{Q}^k| \geq 8^{n-1}$ such that for every $Q \in \mathcal{Q}^k$ there exists $\h{F}_Q \subset h(Q)$, $\H{2}(\h{F}_Q) \geq 2^{-2n-7}$, such that for every $(x^*, y^*) \in \h{F}_Q$ there exist $\eta^1, \xi, \eta^2 \in Q \cap P_{x^*, y^*}$ satisfying that $\xi$ belongs to the line segment bounded by $\eta^1$ and $\eta^2$ and, moreover,
\begin{equation}
\label{eq: BV* g_k spike estimate}
g_k(\xi) - \max\{g_k(\eta^1), g_k(\eta^2)\} \geq 2^{-n-3}.
\end{equation}
Then $$\lim_{k \to \infty} |Dg_k|(Q_0) = \infty.$$
\end{lemma}

\begin{proof}
Fix $q \in \N$, $n = n_q \in \mathcal{N}$ and $k \geq k_n$. Define $$F^{n, k} = \{(x^*, y^*) \in (-2, 2)^2: |\{Q \in \mathcal{Q}^k: (x^*, y^*) \in \h{F}_Q\}| \geq 2^{n-15}\}.$$ Every $(x^*, y^*) \in F^{n, k}$ represents a ``vertical'' line in the cube $Q_0$ that we consider ``good'' for fixed $k$ and $n$, i.e., there are enough cubes $Q \in \mathcal{Q}^k$ satisfying $(x^*, y^*) \in \h{F}_Q \subset h(Q)$. We now prove that there are enough ``good'' vertical lines, more precisely,
\begin{equation}
\label{eq: BV* F^{n, k} estimate}
\H{2}(F^{n, k}) \geq 2^{-15}.
\end{equation}
In fact, consider the space $\Theta = \{1, \ldots, 8^{n+1}\} \times (-2, 2)^2$ with the measure $\mu$ defined as the product of the counting measure on $\{1, \ldots, 8^{n+1}\}$ and the Lebesgue measure $\Ll^2$ on the space $(-2, 2)^2$. Let $E \subset \Theta$ be a measurable subset such that $(i, x^*, y^*) \in E$ if and only if $Q^n_i \in \mathcal{Q}^k$ and $(x^*, y^*) \in \h{F}_{Q^n_i}$. From the assumptions, we know that
\begin{equation}
\label{eq: BV* mu E estimate}
 \begin{aligned}
    \mu(E) &= \sum_{Q \in \mathcal{Q}^k} \H{2}(\h{F}_Q)\\
    &\geq \sum_{Q \in \mathcal{Q}^k} 2^{-2n-7}\\
    &\geq 8^{n-1} 2^{-2n-7} = 2^{n-10}.
 \end{aligned}
\end{equation}
Now, suppose on the contrary that $\H{2}(F^{n, k}) < 2^{-15}$. For every $(x^*, y^*) \in F^{n, k}$, we observe that the number of cubes from $\mathcal{Q}^k$ that $P_{x^*, y^*}$ intersects is, at most, the side length of $Q_0$ divided by the side length of each cube in the dyadic partition $\{Q^n_i\}_{i=1}^{8^{n+1}}$. This implies $$|\{Q \in \mathcal{Q}^k: (x^*, y^*) \in \h{F}_Q\}| \leq \frac{4}{2^{1-n}} = 2^{n+1}.$$ From the definition of $F^{n, k}$, we have for every $(x^*, y^*) \notin F^{n, k}$ $$|\{Q \in \mathcal{Q}^k: (x^*, y^*) \in \h{F}_Q\}| < 2^{n-15}.$$ Combining the two previous estimates for both $F^{n,k}$ and its complement in $(-2,2)^2$, we get
$$\mu(E) < 2^{-15} 2^{n+1} + 4^2 2^{n-15} < 2^{n-10},$$ which contradicts \eqref{eq: BV* mu E estimate}. Therefore, we obtain \eqref{eq: BV* F^{n, k} estimate}.

Next, we fix $q \in \N$ satisfying \eqref{eq: BV* N cap 1...n_q estimate} large so that
\begin{equation}
\label{eq: BV* lambda_q estimate}
    \frac{n_q-1500}{100} > \frac{n_q}{128}.
\end{equation} Let $k \geq k_{n_q}$. Define 
\begin{equation}
\label{eq: F_k definition}
F_k = \{(x^*, y^*) \in (-2, 2)^2: |\{j \in \{16, \ldots, n_q\}: (x^*, y^*) \in F^{j, k}\}| \geq 2^{-30}n_q\}.
\end{equation} This set represents ``better'' vertical lines for a fixed $k$, i.e., every $(x^*, y^*) \in F_k$ is ``good'' for a lot of $n \leq n_q$. We once again prove that there are enough ``better'' vertical lines, that is,
\begin{equation}
\label{eq: BV* F_k estimate}
\H{2}(F_k) \geq 2^{-30}.
\end{equation}
Consider the space $\Pi = \{16, \ldots, n_q\} \times (-2, 2)^2$ with the measure $\nu$ to be, once again, the product of the counting measure on $\{16, \ldots, n_q\}$ and the Lebesgue measure $\Ll^2$ on $(-2, 2)^2$. Let $E \subset \Pi$ be a measurable subset such that $(j, x^*, y^*) \in E$ if and only if $(x^*, y^*) \in F^{j, k}$. Because of \eqref{eq: BV* F^{n, k} estimate}, \eqref{eq: BV* N cap 1...n_q estimate} and \eqref{eq: BV* lambda_q estimate}, we see that
\begin{equation}
\label{eq: BV* nu(E) estimate}
\nu(E) \geq 2^{-15} (|\mathcal{N} \cap \{1, \ldots, n_q\}| - 15) \geq 2^{-15} \frac{n_q - 1500}{100} \geq 2^{-22}n_q.
\end{equation}
Hence, if $\H{2}(F_k) < 2^{-30}$, then by \eqref{eq: F_k definition} we get $$\nu(E) < 2^{-30}n_q + 4^2 2^{-30}n_q < 2^{-22}n_q,$$ which contradicts \eqref{eq: BV* nu(E) estimate}. This implies \eqref{eq: BV* F_k estimate}.

Finally, let $q \in \N$ satisfy \eqref{eq: BV* N cap 1...n_q estimate} and for $k \geq k_{n_q}$ and $(x^*, y^*) \in F_k$ let $g_k^*: [-2, 2] \to \R$ be defined as $$g_k^*(t) = g_k(x^*, y^*, t).$$ This is a one-dimensional function on which we want to apply Lemma \ref{lem: Mountain Range}. In order to do so, set $$J_k = \{j \in \{16, \ldots, n_q\}: (x^*, y^*) \in F^{j,k}\}$$ and observe that \eqref{eq: F_k definition} implies $|J_k| \geq 2^{-30}n_q$. Fix $j \in J_k$. We have $(x^*, y^*) \in F^{j, k}$ and, thus, $$|\{Q \in \mathcal{Q}^k: (x^*, y^*) \in \h{F}_Q\}| \geq 2^{j-15}.$$ Notice that we can enumerate the cubes $Q$ in $\mathcal{Q}^k$ by numbers $i \in \{-2^{j+2}+1, \ldots, 0, \ldots, 2^{j+2}\}$ in a way that $$P_{x^*, y^*} \cap Q_i = \{(x^*, y^*, t): t \in ((i-1)2^{1-j}, i2^{1-j})\}.$$ We set $I^j_k$ as the set of indices $i$ where $(x^*, y^*) \in \h{F}_{Q_i}$. In order to use Lemma \ref{lem: Mountain Range}, we set
\begin{align*}
    J &:= J_k\\
    I^j &:= I^j_k.
\end{align*}
We already checked $|J| \geq 2^{-30}n_q$ and for every $j \in J$ the estimate $|I^j| \geq 2^{j-15}$. Because of \eqref{eq: BV* g_k spike estimate} we can find for every $i \in I^j$ appropriate $$(i-1) 2^{1-j} < \eta^1_{i,j} <\xi_{i, j} < \eta^2_{i,j} < i 2^{1-j}$$ to satisfy the condition in \eqref{eq: BV* I^j definition}.
Thus, we apply Lemma \ref{lem: Mountain Range} to every function $g := g_k^*$ for all $(x^*, y^*) \in F_k$ and obtain by Theorem \ref{thm: BV total variation decomposition}, \eqref{eq: BV directional leq total} and \eqref{eq: BV* F_k estimate}
$$|Dg_k|(Q_0) \geq \int_{F_k} |Dg_k^*|([-2, 2]) \, d\H{2} \geq 2^{-30} 2^{-50} n_q$$ for every $k \geq k_{n_q}$. Due to \eqref{eq: BV* N cap 1...n_q estimate} we have this estimate for infinitely many distinct natural numbers $n_q$ and, therefore, we get $$\lim_{k \to \infty} |Dg_k|(Q_0) = \infty.$$
\end{proof}

\subsection{The sign of the Jacobian - proof}
\label{section: BV* Jacobian proof}
Now, we are ready to tackle the main theorem. Our goal is to find a suitable cube $Q_0 \subset \Omega$ and a subsequence of $f_k$ (which we denote the same) such that for some $\iota \in \{1, 2, 3\}$ the coordinate functions $g_k(x) := f_k^\iota(x)$ (or $g_k(x) := - f_k^\iota(x)$) satisfy the assumptions of Lemma \ref{lem: BV limit Dgk unbounded}, which yields $$\sup_{k \in \N} |Df_k|(Q_0) = \infty.$$ However, Uniform Boundedness Principle says that $f_k \xrightharpoonup{*}f$ implies $$\sup_{k \in \N} |Df_k|(\Omega) < \infty.$$ This is the contradiction we reach in the following proof.

It is enough to prove Theorem \ref{thm: BV limit jacobian} for $n=3$ as, for every homeomorphism $f_k = (f_k^1, f_k^2) \in BV(\Omega, \R^2)$, we can define a homeomorphism $h_k \in BV(\Omega \times \R, \R^3)$ with the same properties as $f_k$ by setting $$h_k(x_1, x_2, x_3) = (f_k^1(x_1, x_2), f_k^2(x_1, x_2), x_3).$$ Suppose $f_k \xrightharpoonup{*} f$ in $BV$. Then clearly $h_k \xrightharpoonup{*} h$ in $BV$, where $$h(x_1, x_2, x_3) = (f^1(x_1, x_2), f^2(x_1, x_2), x_3).$$

\begin{proof}[Proof of Theorem \ref{thm: BV limit jacobian} for $n=3$]
On the contrary, suppose that there is a set of positive measure $E \subset \Omega$ where $J_f < 0$.

\subsection*{Step 1: Choice of $\h{E}$ and $\delta > 0$}

Let $0 < \varepsilon < \frac{1}{10}$, $\mathcal{M}:=\left\{M \in \Q^{3 \times 3}: \det M < 0\right\}$, and for a fixed matrix $M \in \mathcal{M}$ set 
\begin{equation}
\label{eq: BV* E_Q def}
E_M := \left\{x \in E: \abs{\nabla f(x) - M} < \varepsilon \right\}.
\end{equation}
Clearly $E = \bigcup_{M \in \mathcal{M}} E_M$ is a countable union and, therefore, we can find $M$ such that the set $E_M$ has positive measure. Let
\begin{equation}
\label{eq: BV* matrix I}
R = \begin{pmatrix}
    1 & 0 & 0\\
    0 & 1 & 0\\
    0 & 0 & -1
\end{pmatrix}.
\end{equation}
In particular, we want the determinant of $R$ to be negative and $R$ to be an isometry. By a suitable sense-preserving linear change of variables, we can, without loss of generality, assume $M = R$. Indeed, the positivity of $|E_M|$ as well as the degree of $f_k$ is preserved under a linear change of variables with a positive Jacobian.

From Theorem \ref{thm: L1 diff} we deduce that almost every $x_0 \in E_R$ satisfies
\begin{equation*}
\lim_{r \to 0_+} \dashint_{Q(x_0, 4r)} \frac{\abs{f(x)-f(x_0)-\nabla f(x_0)(x-x_0)}}{r}\,dx = 0.
\end{equation*}
Without loss of generality, assume that this is true for every $x_0 \in E_R$. Let $\delta \in (0, 10^{-10})$. Then for every $x_0 \in E_R$ there exists $r_{x_0} > 0$ such that for every $r \leq r_{x_0}$
\begin{equation}
\label{eq: BV* approx diff delta}
    \int_{Q(x_0, 4r)} |f(x) - f(x_0) - \nabla f(x_0)(x-x_0)| \, dx < \delta^3 r^4.
\end{equation}
Then there exist $\h{E} \subset E_R$, $|\h{E}| > 0$, and $r_0 > 0$ such that for every $x_0 \in \h{E}$ we have $r_{x_0} \geq r_0$. This means that for every $x_0 \in \h{E}$ and every $r \leq r_0$ we have \eqref{eq: BV* approx diff delta}. Now, take a density point $x_0 \in \h{E}$ and find $\rho > 0$ such that for every $r \leq \rho$
\begin{equation}
\label{eq: BV* point of density}
|Q(x_0, 4r) \cap \h{E}| > \frac{999}{1000} |Q(x_0, 4r)|.
\end{equation}
Without loss of generality, we assume $\rho \leq r_0.$ Find $n_0 \in \N$ such that $2^{-n_0} \leq r_0$, fix any $n \geq n_0$ and denote $r_n = 2^{-n}$.

Since $f_k \xrightharpoonup{*} f$ in $BV(\Omega, \R^3)$, we also have $f_k \to f$ in $L^1(\Omega, \R^3)$. Thus, we can find $k_n \in \N$ such that
\begin{equation}
\label{eq: BV* choice of k_n}
    \text{for every } k \geq k_n \text{ we have }\int_\Omega |f_k(x) - f(x)| \, dx < \delta^3 r_n^4.
\end{equation}
Combining this with \eqref{eq: BV* approx diff delta} and using the triangle inequality, we conclude that for every $x_0 \in \h{E}$ and every $k \geq k_n$
\begin{equation}
\label{eq: BV* f_k to f in L^1}
    \int_{Q(x_0, 4r_n)} |f_k(x) - f(x_0) - \nabla f(x_0)(x-x_0)| \, dx < 2\delta^3 r_n^4.
\end{equation}
Up to a translation and scaling with a constant $\frac{1}{r_0},$ we assume without loss of generality that $$Q_0 = Q(x_0, 4r_0) = Q(0,4)$$ and $n_0 = 1$.

\subsection*{Step 2: Choice of $y_Q$'s \& Claim}
Consider for a fixed $n \geq n_0$ the dyadic decomposition of $Q_0$ into cubes of side length $2r_n = 2^{1-n}$, which we denote by $\{Q^n_i\}_{i=1}^{8^{n+1}}$. Recall that for a cube $Q \subset \R^3$ we denote its center point by $s(Q)$. First, we claim that at least in $\frac{3}{4}$ of the dyadic cubes $\{Q^n_i\}_{i = 1}^{8^{n+1}}$ there exists some point $y_{Q^n_i} \in \h{E} \cap Q(s(Q^n_i), \frac{r_n}{8})$. If this were not the case, then we would get
\begin{align*}
|Q(0, 4) \setminus \h{E}| &\geq \frac{1}{4} 8^{n+1} \left(\frac{r_n}{8}\right)^3\\
&= \frac{1}{4} 8^{n+1} 2^{-3n} 8^{-3} = \frac{1}{4} \frac{1}{64} > \frac{1}{1000},
\end{align*}
which contradicts \eqref{eq: BV* point of density}. Denote $$\mathcal{Q}_n = \{Q^n_i: i \in \{1, \ldots, 8^{n+1}\} \text{ and there exists } y_{Q^n_i} \in Q(s(Q^n_i), \frac{r_n}{8}) \cap \h{E}\}.$$ We calculated that 
\begin{equation}
\label{eq: BV* good cubes}
|\mathcal{Q}_n| \geq \frac{3}{4} 8^{n+1}.
\end{equation}

Recall the pair of linked squares $L^r_p$ as in Figure \eqref{eq: Square link sides} defined by a side length $r > 0$ and a point $p \in \R^3$. Notice that each of the 8 sides is denoted by a letter from the set $\{a, b, c, d, e, f, g, h\}$. For a side $w \in \{a, \ldots, h\}$, recall the notation $w^r_p$ from \eqref{eq: Square link sides}.

We claim that there exist a subsequence of $\{f_k\}_{k=1}^\infty$, a component $\iota \in \{1, 2, 3\}$, a side $w \in \{a, b, \ldots, h\}$, an increasing sequence of natural numbers $\mathcal{N} = \{n_q\}_{q=1}^\infty$ that satisfies 
\begin{equation}
\label{eq: BV* def N_q}
    \text{for infinitely many $q \in \N$ the estimate }|\mathcal{N} \cap \{1, \ldots, n_q\}| \geq \frac{n_q}{100}
\end{equation} and for every $n = n_q \in \mathcal{N}$ there exists $k_n \in \N$ such that for every $k \geq k_n$ the following holds. For $Q \in \mathcal{Q}_n$ let 
\begin{equation}
\label{eq: BV* choice of G+}
\begin{aligned}
G_{w, k, n}^{Q, \iota, +} := \{ p &\in Q(s(Q), r_n): \text{there exists } \xi \in w_p^{r_n} \text{ such that }\\
&f_k^{\iota}(\xi) - f^\iota(y_Q) - \nabla f^\iota (y_Q) (\xi - y_Q) \geq \frac{r_n}{7} \}
\end{aligned}
\end{equation}
and
\begin{equation*}
\begin{aligned}
G_{w, k, n}^{Q, \iota, -} := \{ p &\in Q(s(Q), r_n): \text{there exists } \xi \in w_p^{r_n} \text{ such that }\\ &f_k^{\iota}(\xi) - f^\iota(y_Q) - \nabla f^\iota (y_Q) (\xi - y_Q) \leq -\frac{r_n}{7}\}.
\end{aligned}
\end{equation*}
Now, let 
\begin{equation}
\label{eq: BV* Tilde Q_a definition}
\mathcal{Q}_{w, k, n}^{\iota, +} = \{Q \in \mathcal{Q}_n: \Ll^{3}(G_{w, k, n}^{Q, \iota, +}) \geq \frac{r_n^3}{48}\}
\end{equation} and we analogously define $\mathcal{Q}_{w, k, n}^{\iota, -}$. The claim asserts that
\begin{equation}
\label{eq: BV* final cubes}
    \max \{|\mathcal{Q}_{w, k, n}^{\iota, +}|, |\mathcal{Q}_{w, k, n}^{\iota, -}|\} \geq 8^{n-1}.
\end{equation}

\subsection*{Step 3: Proof of Claim}
Let $n \geq n_0 = 1$ and fix $Q \in \mathcal{Q}_n$. By translating both $Q$ and $Q_0$, we assume without loss of generality that $Q = Q(0, 2r_n)$. We find $y_Q \in Q(0, \frac{r_n}{8})$. Now consider any point $p \in Q(0, r_n)$, and define a pair of linked squares $$L^{r_n}_p = S^{r_n}_p \cup T^{r_n}_p$$ in the same manner as in \eqref{eq: Square link sides}.
It is easy to observe that $L^{r_n}_p \subset Q \cap Q(y_Q, 4r_n)$ and, due to \eqref{eq: Distance Linked Squares}, $\dist(S^{r_n}_p, T^{r_n}_p) = \frac{r_n}{2}$.

We claim that there exists a point $\xi^{r_n}_{p,k} \in L^{r_n}_p$ such that 
\begin{equation}
\label{eq: BV* xi_p oscillation estimate}
|f_k(\xi^{r_n}_{p,k}) - f(y_Q) - \nabla f(y_Q)(\xi^{r_n}_{p,k} - y_Q)| \geq \frac{r_n}{4}.
\end{equation} If this were not true, we could construct a homotopy between the mapping $f_k(x)$ and $f(y_Q) + \nabla f(y_Q)(x - y_Q)$ similarly to the construction after \eqref{eq: BV bad links} and we would obtain a contradiction with Propositions \ref{prop: Stability under homotopy} and \ref{prop: Linking number squares}. Here we use $|\nabla f(y_Q) - R| < \varepsilon$ and the fact that $R$ has a negative determinant and is an isometry.

Now, let $w \in \{a, \ldots, h\}$ denote a side of the link $L^{r_n}_p$ as in \eqref{eq: Square link sides}. Let $\iota \in \{1, 2, 3\}$. Denote $G_{w,k,n}^{Q, \iota, +}$ the set of points $p \in Q(0, r_n)$ where $\xi^{r_n}_{p,k} \in w^{r_n}_p$ and, moreover, 
\begin{equation}
\label{eq: BV* xi_p positive oscillation iota coordinate estimate}
f_k^\iota (\xi_{p,k}^{r_n}) - f^\iota(y_Q) - \nabla f^\iota(y_Q)(\xi_{p,k}^{r_n} - y_Q) \geq \frac{r_n}{7}.
\end{equation}
Analogously, we define $G_{w, k, n}^{Q, \iota, -}$ as the set of points $p \in Q(0, r_n)$ where $\xi_{p,k}^{r_n} \in w^{r_n}_p$ and $$f_k^{\iota}(\xi_{p,k}^{r_n}) - f^\iota(y_Q) - \nabla f^\iota (y_Q) (\xi_{p,k}^{r_n} - y_Q) \leq -\frac{r_n}{7}.$$ Notice that, because every $f^\iota$ and $f^\iota_k$ is measurable, the sets $G^{Q, \iota, +}_{w, k, n}$ and $G^{Q, \iota, -}_{w, k, n}$ are measurable.

We prove that for every $p \in Q(0, r_n)$ there exist $w \in \{a, \ldots, h\}$ and $\iota \in \{1, 2, 3\}$ such that $$p \in G^{Q, \iota, +}_{w, k, n} \cup G^{Q, \iota, -}_{w, k, n}.$$ In fact, if this were not the case, we would arrive at a contradiction with \eqref{eq: BV* xi_p oscillation estimate} by
\begin{align*}
\frac{r_n}{4} &\leq |f_k(\xi_{p,k}^{r_n}) - f(y_Q) - \nabla f(y_Q)(\xi_{p,k}^{r_n} - y_Q)|\\
&= \sqrt{\sum_{\iota = 1}^3 |f^\iota_k(\xi_{p,k}^{r_n}) - f^\iota(y_Q) - \nabla f^\iota(y_Q)(\xi_{p,k}^{r_n} - y_Q)|^2}\\
&< \sqrt{3 \frac{r_n^2}{49}} < r_n \sqrt{\frac{3}{48}} = \frac{r_n}{4}.
\end{align*}

There are $8$ choices for $w$, $3$ choices for $\iota$ and $2$ choices for the sign, i.e. $+$ or $-$. Recall that for every $n \geq n_0 = 1$ there is $k_n$ from \eqref{eq: BV* choice of k_n}. We inductively claim that for every $n \geq n_0 = 1$ there exist a side $w_n$, a component $\iota_n$ and a sign $s_n \in \{+, -\}$ such that there is an infinite set of indices $$\mathcal{K}_n \subset \{k_n, k_n + 1, \ldots \} \cap \mathcal{K}_{n-1}$$ where 
\begin{equation}
\label{eq: BV* better cubes def}
\mathcal{Q}_{k, n} = \{Q \in \mathcal{Q}_n: \mathcal{L}^3(G^{Q, \iota_n, s_n}_{w_n, k, n}) \geq \frac{r_n^3}{48}\}
\end{equation}
satisfies for every $k \in \mathcal{K}_n$ 
\begin{equation}
\label{eq: BV* better cubes estimate}
|\mathcal{Q}_{k,n}| \geq \frac{1}{48} |\mathcal{Q}_n| \geq 8^{n-1}.
\end{equation}
Let us begin the inductive process by setting $\mathcal{K}_0 = \emptyset$.

For $n = n_0 = 1$ notice that there are in total $48$ choices for a side $w$, a component $\iota$ and a sign $+$ or $-$. Thus, for every $k \geq k_1$ and every $Q \in \mathcal{Q}_1$ we can find a side $w$, a component $\iota$ and (without loss of generality) a sign $+$ such that $$\mathcal{L}^3(G^{Q, \iota, +}_{w, k, n}) \geq \frac{r_n^3}{48}.$$ Clearly at least one of the $48$ choices occurs infinitely many times. It follows that we can find suitable $\mathcal{Q}_{k,1}$, $w_1$, $\iota_1$ and $s_1$ such that there is an infinite set $$\mathcal{K}_1 := \{ k \geq k_1: \text{\eqref{eq: BV* better cubes estimate} is true for } \mathcal{Q}_{k,1} \}.$$ For the inductive step, suppose that we found a suitable infinite set of indices $\mathcal{K}_n$ for some $n \in \N$. Then we notice that $$\{k_{n+1}, k_{n+1} + 1, \ldots\} \cap \mathcal{K}_n$$ is infinite as well. Thus, using the same reasoning as above, we conclude that we can find suitable $\mathcal{Q}_{k,n+1}$, $w_{n+1}$, $\iota_{n+1}$ and $s_{n+1}$ (that might not coincide with $w_n$, $\iota_n$ and $s_n$!) such that there is an infinite set $$\mathcal{K}_{n+1} := \{ k \in \{k_{n+1}, k_{n+1}+1, \ldots\} \cap \mathcal{K}_n: \text{\eqref{eq: BV* better cubes estimate} is true for } \mathcal{Q}_{k,n+1} \}.$$ This concludes the inductive process.

Now, for a fixed $n \geq n_0 = 1$ and the chosen $w_n$, $\iota_n$ and $s_n$ let us change the notation $$\mathcal{K}_{w_n, n}^{\iota_n, s_n} := \mathcal{K}_n$$ and if a triplet of a side $w$, a component $\iota$ and a sign $s \in \{+, -\}$ satisfies $(w, \iota, s) \neq (w_n, \iota_n, s_n)$, we set $$\mathcal{K}_{w, n}^{\iota, s} := \emptyset.$$ Then for a side $w$ and a component $\iota$ let $$\mathcal{N}_w^{\iota, +} := \{ n \geq n_0 = 1:  \mathcal{K}_{w, n}^{\iota,+}  \text{ is infinite} \}$$ and $$\mathcal{N}_w^{\iota, -} \{ n \geq n_0 = 1:  \mathcal{K}_{w, n}^{\iota,-}  \text{ is infinite} \}.$$ Notice that every $n \geq n_0 = 1$ satisfies $$n \in \mathcal{N}_{w_n}^{\iota_n, s_n}.$$ Furthermore, we set 
\begin{equation}
\label{eq: BV* Lambda definition}
\Lambda_w^{\iota, +} := \{ n \in \mathcal{N}_w^{\iota, +}: |\mathcal{N}_w^{\iota, +} \cap \{1, \ldots, n\}| \geq \frac{n}{48} \}
\end{equation}
and $$\Lambda_w^{\iota, -} := \{ n \in \mathcal{N}_w^{\iota, -}: |\mathcal{N}_w^{\iota, -} \cap \{1, \ldots, n\}| \geq \frac{n}{48} \}.$$ Recall that there are $48$ possibilities for the choice of a side $w$, a component $\iota$ and a sign $+$ or $-$. It follows that we can find a side $w$, a component $\iota$ and (without loss of generality) a sign $+$ such that $\Lambda_w^{\iota, +}$ is infinite.
Then we denote and enumerate $$\mathcal{N} = \{n_q\}_{q=1}^{\infty} := \mathcal{N}_w^{\iota, +}$$ and choose a subsequence $$\{f_k: k = \min \mathcal{K}^{\iota, +}_{w, n}\}_{n \in \mathcal{N}^{\iota, +}_w}.$$ This concludes the proof of the Claim from \textbf{Step 2}. Indeed, we chose the subsequence $\{f_k\}_{k=1}^{\infty}$ in a way that the estimate in \eqref{eq: BV* def N_q} is true for every $n_q \in \Lambda^{\iota, +}_w$ thanks to \eqref{eq: BV* Lambda definition} and $\Lambda^{\iota, +}_w$ is an infinite subset of $\mathcal{N}$. Moreover, the inequality \eqref{eq: BV* final cubes} follows from \eqref{eq: BV* better cubes estimate} as $\mathcal{Q}_{k, n}$ from \eqref{eq: BV* better cubes def} corresponds to $\mathcal{Q}_{w, k, n}^{\iota, +}$ from \eqref{eq: BV* Tilde Q_a definition}.

\subsection*{Step 4: Application of Lemma \ref{lem: BV limit Dgk unbounded}}
Without loss of generality, pick the side $w = a$, the component $\iota = 1$ and the sign $+$ in the statement of Claim in \textbf{Step 2}. Recall the family of cubes $\mathcal{Q}_{a,k,n}^{1,+}$ from \eqref{eq: BV* Tilde Q_a definition}. Let us return to $Q = Q(0, 2r_n) \in \mathcal{Q}_{a,k,n}^{1, +}$, let $\bar{Q} = Q(0, r_n)$ and $h(\bar{Q}) \subset \R^2$ be its upper face, in our setting $h(\bar{Q}) = (-\frac{r_n}{2}, \frac{r_n}{2})^2$. Recall \eqref{eq: BV* choice of G+} and define 
$$F_{Q} = \{(x^*, y^*) \in h(\bar{Q}): \H{1}(G_{a,k,n}^{Q, 1, +} \cap P_{x^*, y^*}) \geq \frac{r_n}{100}\}.$$
This set represents the ``good'' vertical lines, where there is a lot of oscillation of the first component $f_k^1$ on a fixed scale $2^{-n}$ inside the dyadic cube $Q$. 
We prove that there is enough of these ``good'' vertical lines, that is,
\begin{equation}
\label{eq: BV* F_Q estimate}
\H{2}(F_Q) \geq \frac{r_n^2}{100}.
\end{equation}
If this were not the case, then Fubini theorem yields
\begin{align*}
\Ll^{3}(G_{a,k,n}^{Q, 1, +}) &= \int_{h(\bar{Q})} \H{1}(G_{a,k,n}^{Q, 1, +} \cap P_{x^*, y^*}) d\H{2}((x^*, y^*))\\
&\leq \frac{1}{100}r_n^2 r_n + r_n^2 \frac{1}{100} r_n < \frac{r_n^3}{48},
\end{align*}
which contradicts the fact that $Q \in \mathcal{Q}_{a,k,n}^{1, +}$ and \eqref{eq: BV* Tilde Q_a definition}.

Next, we want the set of ``bad'' vertical lines, where the mapping $f_k$ is far from the affine map $f(y_Q) + \nabla f(y_Q)(\xi - y_Q)$ inside the cube $Q$ on a relatively big portion of the line, to have small measure. Thus, we define
\begin{equation*}
\begin{aligned}
\Tilde{F}_Q &= \{(x^*, y^*) \in h(Q):\\
&\H{1}(\{\xi \in Q \cap P_{x^*, y^*}: |f_k(\xi) - f(y_Q) - \nabla f(y_Q)(\xi - y_Q)| \geq \delta r_n\}) \geq \delta r_n\}
\end{aligned}
\end{equation*}
and claim that 
\begin{equation}
\label{eq: BV* Tilde F_Q estimate}
\H{2}(\Tilde{F}_Q) < 4\delta r_n^2.
\end{equation}
Otherwise, we would obtain by Fubini theorem
\begin{align*}
&\int_{Q(y_Q, 4r_n)} |f_k(x) - f(y_Q) - \nabla f(y_Q)(x - y_Q)| \, dx\\
&\geq \int_{h(Q)} \int_{Q \cap P_{x^*, y^*}} |f_k(\xi) - f(y_Q) - \nabla f(y_Q)(\xi - y_Q)| \, d\H{1}(\xi)d\H{2}((x^*, y^*))\\
&\geq \int_{\Tilde{F}_Q} \delta r_n \delta r_n \, d\H{2} \geq 4\delta^3 r_n^4,
\end{align*} which contradicts \eqref{eq: BV* f_k to f in L^1}.

Hence, we can define the set representing the ``better'' vertical lines $\h{F}_Q := F_Q \setminus \Tilde{F}_Q \subset h(Q)$. Combining \eqref{eq: BV* F_Q estimate} and \eqref{eq: BV* Tilde F_Q estimate} yields the estimate
\begin{equation}
\label{eq: BV* F_Q correct estimate}
\H{2}(\h{F}_Q) \geq (\frac{1}{100} - 4 \delta)r_n^2 \geq 2^{-7}r_n^2 = 2^{-2n-7}.
\end{equation}
We now show that $\h{F}_Q$ is the correct subset of the upper face $h(Q)$ that we want to consider in Lemma \ref{lem: BV limit Dgk unbounded}. To be more precise, we show that for every $(x^*, y^*) \in \h{F}_Q$ we can find $\eta^1, \xi, \eta^2 \in Q \cap P_{x^*, y^*}$ such that
\begin{itemize}
\label{eq: BV* properties of F_Q}
    \item $|\eta^2 - \eta^1| < \delta r_n$,
    \item $\xi$ belongs to the line segment bounded by $\eta^1$ and $\eta^2$,
    \item $f^1_k(\xi) - f^1(y_Q) - \nabla f^1(y_Q)(\xi - y_Q) \geq \frac{r_n}{7}$,
    \item $|f_k(\eta^i) - f(y_Q) - \nabla f(y_Q)(\eta^i - y_Q)| < \delta r_n$ for $i \in \{1, 2\}$.
\end{itemize}
To prove this, consider $(x^*, y^*) \in \h{F}_Q$. From the definition of $\h{F}_Q$, we see that $$G_{a,k,n}^{Q, 1, +} \cap P_{x^*, y^*} \cap Q(0, \frac{999}{1000} r_n) \neq \emptyset$$ and, thus, there exists $p \in Q(0, \frac{999}{1000}r_n) \cap P_{x^*, y^*}$ such that the point $\xi_{p,k}^{r_n}$ belongs to $a^{r_n}_p$. From the definition, we see that $a^{r_n}_p \subset P_{x^*, y^*} \cap Q(0, \frac{1999}{1000}r_n)$, which implies $$\xi_{p,k}^{r_n} \in P_{x^*, y^*} \cap Q(0, \frac{1999}{1000} r_n).$$ Further, $(x^*, y^*) \notin \Tilde{F}_Q$, which means that $$\H{1}(\{\eta \in Q \cap P_{x^*, y^*}: |f_k(\eta) - f(y_Q) - \nabla f(y_Q)(\eta - y_Q)| \geq \delta r_n\}) < \delta r_n.$$ Since $\delta < 10^{-10}$, we can find $\eta^1, \eta^2 \in Q \cap P_{x^*, y^*}$ such that $\xi = \xi_{p,k}^{r_n}$ belongs to the line segment bounded by $\eta^1$ and $\eta^2$, $|f_k(\eta^i) - f(y_Q) - \nabla f(y_Q)(\eta^i - y_Q)| < \delta r_n$ for $i \in \{1, 2\}$ and $|\eta^2 - \eta^1| < \delta r_n$.

Finally, we check that the (sub)sequence $$g_k = f^1_k$$ satisfies the assumptions of Lemma \ref{lem: BV limit Dgk unbounded}, which means that $f^1_k$ are highly oscillating and, therefore, so are $f_k$. Let $\mathcal{N} = \{n_q\}_{q=1}^{\infty}$ be as in the Claim in \textbf{Step 2} and let $q \in \N$ satisfy \eqref{eq: BV* def N_q}. Let $n = n_q \in \mathcal{N}$. We take $k_n \in \N$ given by \eqref{eq: BV* f_k to f in L^1} and consider any $k \geq k_n$. Now, recall \eqref{eq: BV* Tilde Q_a definition} and denote $$\mathcal{Q}^k := \mathcal{Q}_{a,k,n}^{1, +}.$$ Next, \eqref{eq: BV* final cubes} implies $$|\mathcal{Q}^k| \geq 8^{n-1}.$$ For every $Q \in \mathcal{Q}^k$, we found a set $\h{F}_Q \subset h(Q)$ satisfying, due to \eqref{eq: BV* F_Q correct estimate}, the estimate $$\H{2}(\h{F}_Q) \geq 2^{-2n-7}.$$ From the list above, we also see that for every $(x^*, y^*) \in \h{F}_Q$ we found candidates for $\eta^1, \xi, \eta^2 \in Q \cap P_{x^*, y^*}$. It only remains to show that $$g_k(\xi) - \max\{g_k(\eta^1), g_k(\eta^2)\} \geq 2^{-n-3}.$$ Let $i \in \{1, 2\}$. We have by the triangle inequality and \eqref{eq: BV* xi_p positive oscillation iota coordinate estimate}
\begin{align*}
f^1_k(\xi) - f^1_k(\eta^i)
&= f^1_k(\xi) - f^1(y_Q) - \nabla f^1(y_Q)(\xi - y_Q)\\
&\phantom{=} - f^1_k(\eta^i) + f^1(y_Q) + \nabla f^1(y_Q)(\eta^i - y_Q)\\
&\phantom{=} + \nabla f^1(y_Q)(\xi - \eta^i)\\
&\geq \frac{r_n}{7} - |f_k(\eta^i) - f(y_Q) - \nabla f(y_Q)(\eta^i - y_Q)|\\
&\phantom{=} - (|\nabla f(y_Q) - R||\xi - \eta^i| + |\xi - \eta^i|).
\end{align*}
Recall the choice of $\varepsilon$ and the definition of $E_R$ in \eqref{eq: BV* E_Q def}. Then we use $y_Q \in E_R$, the fact that the matrix $R$ from \eqref{eq: BV* matrix I} is an isometry, and $|\xi - \eta^i| \leq |\eta^2 - \eta^1| < \delta r_n$ to obtain
\begin{align*}
f^1_k(\xi) - f^1_k(\eta^i)
&\geq \frac{r_n}{7} - \delta r_n - \varepsilon \delta r_n - \delta r_n > \frac{r_n}{8} = 2^{-n-3}.
\end{align*}

\subsection*{Step 5: Final contradiction}
In \textbf{Step 4}, we checked  that the (sub)sequence of functions $g_k = f^1_k \in BV(Q_0)$ satisfies the assumptions of Lemma \ref{lem: BV limit Dgk unbounded}. This implies that $$\limsup_{k \to \infty} |Df_k|(Q_0) \geq \limsup_{k \to \infty} |Dg_k|(Q_0) = \infty.$$ However, Uniform Boundedness Principle yields $$\sup_{k \in \N} |Df_k|(\Omega) < \infty.$$ This gives us the desired contradiction.
\end{proof}

\subsection{Differentiability in two dimensions}
\label{section: BV* differentiability}

\begin{proof}[Proof of the Moreover part of Theorem \ref{thm: BV limit jacobian}]

Let $f$ be the weak-* limit of $BV$ homeomorphisms $f_k \in BV(\Omega, \R^2)$. Theorem \ref{thm: BV homeo are diff} tells us that for every $k \in \N$ we have $f_k^{-1} \in BV(f_k(\Omega), \R^2)$. Let $x \in \Omega = f_k^{-1}(f_k(\Omega))$ and $r > 0$ be such that $B(x, 4r) \subset \subset \Omega$. We use Lemma \ref{lem: BV diameter of ball} (with $f_k^{-1}$ in the role of $f$ and $f_k(\Omega)$ instead of $\Omega$) to obtain
\begin{equation*}
r \diam f_k(B(x,r)) \leq C \abs{Df_k^{-1}}(\overline{f_k(B(x, 2r)}).
\end{equation*}
Next, we use Theorem \ref{thm: BV total variation Df and Df^-1 are the same} and infer
\begin{equation}
\label{eq: BV diam ball leq variation}
r \diam f_k(B(x,r)) \leq C \abs{Df_k^{-1}}(f_k(B(x,3r))) = C\abs{Df_k}(B(x,3r)).
\end{equation}

Now, the limit $f$ is defined $\Ll^2$-almost everywhere in $\Omega$. We can work with the representative defined everywhere in $\Omega$ as $$f(x) = \limsup_{r \to 0_+} \dashint_{B(x,r)} f.$$ Then $f_k \to f$ in $L^1(\Omega, \R^2)$ implies (up to a subsequence that we denote the same) $f_k(y) \to f(y)$ for $\Ll^2$-almost every $y \in \Omega$. Observe that
\begin{align*}
\diam f&(B(x,r)) = \sup \{\abs{f(x_1) - f(x_2)}: x_1, x_2 \in B(x,r)\}\\
&= \sup \{\abs{f(x_1) - f(x_2)}: x_1, x_2 \in B(x,r), f_k(x_i) \to f(x_i), i=1,2\},
\end{align*}
as the values of $f$ everywhere in $\Omega$ are defined to be averages of the values $\Ll^2$-almost everywhere and, thus, they are bounded by the range of the values of $f$ $\Ll^2$-almost everywhere.

From the convergence $f_k \to f$ almost everywhere, it is easy to deduce that for almost every $x_1, x_2 \in B(x,r)$ we have $$\abs{f(x_1) - f(x_2)} = \lim_{k \to \infty} \abs{f_k(x_1) - f_k(x_2)}.$$ 
It follows that for every $x \in \Omega$ and $r > 0$ satisfying $B(x, r) \subset \Omega$ we obtain
\begin{equation}
\label{eq: BV limit diam leq liminf diam}
\diam f(B(x, r)) \leq \liminf_{k \to \infty} \diam f_k(B(x, r)).
\end{equation}

Moreover, Uniform Boundedness Principle says that the sequence of total variations $\{\abs{Df_k}(\Omega)\}$ is bounded and, thus, there exists a subsequence (that we denote again $f_k$) and a Radon measure $\mu$ that is a weak-* limit of $\{\abs{Df_k}\}$. Let $x \in \Omega$ be a point where
\begin{equation}
\label{eq: max operator measure finite}
M_\mu(x) := \limsup_{r \to 0_+} \frac{\mu(B(x,r) \cap \Omega)}{\abs{B(x,r)}} < \infty.
\end{equation}
It is known that this value is finite for $\Ll^2$-almost every $x \in \Omega$ for any Radon measure $\mu$.

Now we are prepared for a computation. Let $x \in \Omega$ and $r > 0$ be such that $B(x, 4r) \subset \subset \Omega$ and let $x$ satisfy \eqref{eq: max operator measure finite}. Compiling the above estimates, namely \eqref{eq: BV limit diam leq liminf diam} and \eqref{eq: BV diam ball leq variation}, it follows that
\begin{align*}
\frac{\diam f(B(x,r))}{r} &\leq \liminf_{k \to \infty} \frac{\diam f_k(B(x,r))}{r}\\
&\leq \liminf_{k \to \infty} \frac{C}{\abs{B(x, r)}} \abs{Df_k}(B(x, 3r)).
\end{align*}
Subsequently, we introduce a smooth function $\varphi$ satisfying $\varphi \equiv 1$ on $B(x, 3r)$ and $\spt \varphi \subset \subset B(x, 4r)$ and then we use the characterization of weak-* limits of Radon measures. We finally obtain
\begin{align*}
\frac{\diam f(B(x,r))}{r} &\leq \liminf_{k \to \infty} \frac{C}{\abs{B(x, r)}} \int_{B(x, 4r)} \varphi \, d\abs{Df_k}\\
&\leq \frac{C}{\abs{B(x, r)}} \int_{B(x, 4r)} \varphi \, d\mu\\
&\leq C \frac{\mu(B(x, 4r))}{\abs{B(x, 4r)}}.
\end{align*}
Thanks to \eqref{eq: max operator measure finite} we see that the right-hand-side of the inequality is bounded for $\Ll^2$-almost every $x \in \Omega$ as $r \to 0_+$. We conclude the proof by applying Stepanov theorem.
\end{proof}

\noindent \textbf{Acknowledgements:} The author was supported by the grant GA\v{C}R P201/24-10505S.
The author thanks his supervisor Stanislav Hencl for his valuable insights and Barbora Bene\v{s}ová for her ideas to make the paper easier to read.

\bibliographystyle{abbrv}
\bibliography{biblio_arxiv}

\end{document}